\documentclass[11pt,reqno]{amsart}
\usepackage{ytableau}
\usepackage{amsthm}
\usepackage{amssymb}
\usepackage{latexsym}
\usepackage{multicol}
\usepackage{etoolbox} 
\usepackage{verbatim,enumerate}
\usepackage{accents}
\usepackage[utf8]{inputenc}
\usepackage{hyperref}
\usepackage{amsmath,amscd}
\usepackage{soul}
\usepackage{tikz} 
\usepackage{mathtools}
\usepackage{cleveref}
\usepackage{dirtytalk}
\usepackage{stmaryrd}
\usetikzlibrary{matrix,arrows,decorations.pathmorphing,arrows.meta,calc}
\usepackage{enumitem}
\usepackage{todonotes}
\newlength\cellsize 

\savebox2{%
\begin{picture}(10,10)
\put(0,0){\line(1,0){10}}
\put(0,0){\line(0,1){10}}
\put(10,0){\line(0,1){10}}
\put(0,10){\line(1,0){10}}
\end{picture}}

\newcommand\cellify[1]{\def\thearg{#1}\def\nothing{}%
\ifx\thearg\nothing\vrule width0pt height\cellsize depth0pt%
  \else\hbox to 0pt{\usebox2\hss}\fi%
  \vbox to 10\unitlength{\vss\hbox to 10\unitlength{\hss$#1$\hss}\vss}}

\newlength{\bibitemsep}
\newlength{\bibparskip}
\let\oldthebibliography\thebibliography
\renewcommand\thebibliography[1]{%
  \oldthebibliography{#1}%
  \setlength{\parskip}{\bibitemsep}%
  \setlength{\itemsep}{\bibparskip}%
}

\newcommand\tableau[1]{\vtop{\let\\=\cr
\setlength\baselineskip{-10000pt}
\setlength\lineskiplimit{10000pt}
\setlength\lineskip{0pt}
\halign{&\cellify{##}\cr#1\crcr}}}

\savebox7{
\begin{picture}(10,10)
  \put(5,5){\circle{9.4}}
\end{picture}}

\newcommand{\cirfy}[1]{\def\thearg{#1}\def\nothing{}%
\ifx\thearg\nothing\vrule width0pt height\cellsize depth0pt%
  \else\hbox to 0pt{\usebox7\hss}\fi%
  \vbox to 10\unitlength{\vss\hbox to 10\unitlength{\hss$#1$\hss}\vss}}

\newcommand\cirtab[1]{\vtop{\let\\=\cr
\setlength\baselineskip{-10000pt}
\setlength\lineskiplimit{10000pt}
\setlength\lineskip{0pt}
\halign{&\cirfy{##}\cr#1\crcr}}}

\theoremstyle{plain}
\newtheorem*{prop}{Proposition}
\newtheorem{thm}{Theorem}
\newtheorem*{lemma}{Lemma}
\newtheorem*{cor}{Corollary}

\theoremstyle{definition}
\newtheorem*{example}{Example}
\newtheorem*{definition}{Definition}

\newtheorem*{rem}{Remark}

\theoremstyle{remark}

\newcommand{\lie}[1]{\mathfrak{#1}}

\newcommand\bc{\mathbb C}
\newcommand\bn{\mathbb N}
\newcommand\bz{\mathbb Z}

\newcommand{\tikzcircle}[2][black]{\tikz[baseline=-0.5ex]\draw[#1,radius=#2] (0,0) circle ;}

\advance\textwidth by 1.2in  \advance\oddsidemargin by -.6in \advance\evensidemargin by -.6in

\newcounter{cnt}
 \makeatletter
\def\mydggeometry{\makeatletter\dg@YGRID=1\dg@XGRID=20\unitlength=0.003pt\makeatother}
\makeatother \theoremstyle{remark}

\numberwithin{equation}{section}
\makeatletter
\def\section{\def\@secnumfont{\mdseries}\@startsection{section}{1}%
  \z@{.7\linespacing\@plus\linespacing}{.5\linespacing}%
  {\normalfont\scshape\centering}}
\def\subsection{\def\@secnumfont{\bfseries}\@startsection{subsection}{2}%
  {\parindent}{.5\linespacing\@plus.7\linespacing}{-.5em}%
  {\normalfont\bfseries}}
\def\subsubsection{\def\@subsecnumfont{\bfseries}\@startsection{subsubsection}{3}%
  {\parindent}{.5\linespacing\@plus.7\linespacing}{-.5em}%
  {\normalfont\bfseries}}
  \makeatother

\title[Root generated subalgebras and $\pi$-systems]{Graded embeddings, root generated subalgebras and $\pi$-systems for quasisimple Kac-Moody superalgebras}
\author{Irfan Habib}
\address{Technical University of Munich, TUM School of Computation, Information and Technology, Department of Mathematics, Boltzmannstr. 3, 85748 Garching bei München, Germany}
\email{irfan.habib@tum.de}
\thanks{}
\author{Deniz Kus}
\address{Technical University of Munich, TUM School of Computation, Information and Technology, Department of Mathematics, Boltzmannstr. 3, 85748 Garching bei München, Germany}
\email{deniz.kus@tum.de}
\thanks{D.K. was partially funded by the Deutsche Forschungsgemeinschaft (DFG, German Research Foundation) –  grant 562506224.}
\author{Chaithra Pilakkat}
\address{Department of Mathematics, Indian Institute of Science, Bangalore 560012, India}
\email{chaithrap@iisc.ac.in}
\thanks{C.P. was partially supported by Prime Minister's Research Fellowship (TF/PMRF-22-5467.03).}

\subjclass[2020]{}

\begin{document}
\begin{abstract}
Motivated by a construction of Gorelik and Shaviv, we show that the real roots of a root generated subalgebra associated with a $\pi$-system contained in the positive roots are obtained by successive applications of even and odd reflections to the $\pi$-system, and that they form a real closed subroot system. Using this result, we establish an analogue of Dynkin’s bijection in the setting of symmetrizable quasisimple Kac–Moody superalgebras. In addition, we obtain several results on root strings in the super setting, analogous to those of Billig and Pianzola, and show that graded embeddings arise as root generated subalgebras associated with linearly independent $\pi$-systems.
\end{abstract}
\maketitle

\section{Introduction}
Understanding graded embeddings of Kac–Moody (super)algebras is important both from a structural perspective and due to their applications in mathematical physics. Of particular interest are graded embeddings in which the Chevalley generators are mapped to real root vectors. Such embeddings admit both algebraic and combinatorial interpretations.

For a finite-dimensional semisimple Lie algebra $\mathring{\lie g}$, studying graded embeddings is equivalent to understanding its semisimple subalgebras. In his seminal work \cite{dynkin52semisimple}, Dynkin classified all semisimple subalgebras of $\mathring{\lie g}$ and showed that they are in bijection both with linearly independent $\pi$-systems of $\mathring{\Delta}$ (the root system of $\mathring{\lie g}$) contained in the set of positive roots and with the closed subroot systems of $\mathring{\Delta}$. It is well known that in the finite-dimensional setting, any $\pi$-system consisting of positive roots is automatically linearly independent. Closed subroot systems have been classified via the classification of maximal closed subroot systems, which was carried out in \cite{BorelDeS}. Together, these results provide a complete combinatorial solution to the graded embedding problem for finite-dimensional semisimple Lie algebras.

The finite-dimensional picture does not extend to arbitrary Kac–Moody algebras $\lie g$, and counterexamples already appear for low-rank untwisted affine Lie algebras. In this broader setting, the appropriate replacements for closed subroot systems and semisimple subalgebras are, respectively, real closed subroot systems and root-generated subalgebras; see Section~\ref{secbasicdef} for precise definitions. For general symmetrizable Kac–Moody algebras, foundational results were obtained by Morita \cite{morita1989certain} and Naito \cite{naito1992regular}, who showed that linearly independent $\pi$-systems give rise to graded embeddings. This idea has since been used repeatedly to construct graded embeddings in subsequent works; see, for example, \cite{feingold2004subalgebras,habib2024pisystems,viswanath08embedding}. In the affine case, a more classification-oriented approach was developed in \cite{roy2019maximal}, where the authors classified (maximal) real closed subroot systems of affine Lie algebras and, using this classification, proved that every root-generated subalgebra is generated by a $\pi$-system (see \cite{habib2024maximal,KV21a} for generalizations to affine reflection systems).
Recently, the authors of \cite{idv2023root} unified the results described above and proved that an analogue of Dynkin’s bijection holds for arbitrary symmetrizable Kac–Moody algebras. In this setting, every real closed subroot system admits a unique $\pi$-system (see Section~\ref{secpi} for the definition), and every root-generated subalgebra is generated by its associated $\pi$-system. Moreover, an explicit presentation of these root-generated subalgebras is provided, and they are shown to be of Kac–Moody type.

The motivation of this article is to study graded embeddings in the super setting from a combinatorial perspective, with the broader goal of investigating the structure of Kac–Moody superalgebras. We focus on the root-generated subalgebras of quasi-simple, regular, symmetrizable Kac–Moody superalgebras that are generated by $\pi$-systems. Motivated by a construction of Gorelik and Shaviv \cite{gorelik17generalized}, we determine their set of real roots (see Theorem~\ref{proprealroots}) and show that they form a real closed subroot system. Just as in the Lie algebra case, the properties of root strings play a central role in understanding root-generated subalgebras. We determine the real roots inside a root string for Kac–Moody superalgebras in the spirit of \cite{billig1995root,morita1988root} (see Corollary~\ref{proprootstring}) and prove that in the symmetrizable case a root string contains at most four real roots (see Proposition~\ref{prop32}). Finally, we establish Dynkin’s bijection in the super setting, providing a correspondence between root-generated subalgebras arising from $\pi$-systems, closed subroot systems that admit $\pi$-systems, and $\pi$-systems contained in positive root systems (see Theorem~\ref{dynbij}). In the concluding subsection, we discuss how $\pi$-systems and closed subroot systems arise naturally in the graded embedding problem and outline several applications of these embeddings in mathematical physics.

 However, the definitions and conventions for Kac–Moody superalgebras - such as the definition of the algebras themselves or of real roots - are not uniform in the literature (see, e.g., \cite{gorelik2024rootgroupoid,Serganova11KacMoody,Wa01}). In this article, we adopt the conventions of \cite{Serganova11KacMoody}. Several fundamental properties that hold for Kac–Moody algebras fail in the super setting. For instance, a real closed subroot system may fail to admit a $\pi$-system, and a root-generated subalgebra arising from a $\pi$-system need not be generated by one contained in the positive roots; we provide explicit examples along the way. This makes the study of root-generated subalgebras, and consequently the graded embedding problem, particularly interesting.
Even in the finite-dimensional case, the graded embedding problem remains far from fully understood, and only a few works address regular subalgebras in the super setting. In \cite{vander87regular}, embeddings of graded subalgebras were constructed using specific linearly independent $\pi$-systems, depending on the type. In \cite{Nayak13Embedding}, the analogue of \cite{viswanath08embedding} was established for the hyperbolic Kac–Moody superalgebra $HD(4,1)$ of rank $6$, although their definition uses only the Serre and contragredient relations, rather than the maximal ideal intersecting the Cartan subalgebra trivially.

\noindent\textit{The paper is organized as follows:} In Section~\ref{prel}, we recall basic facts about Kac–Moody superalgebras and introduce the notions used throughout the paper. In Section~\ref{secrootstring}, we study root strings, their symmetries and determine the positions of real roots. In Section~\ref{secrgs}, we analyze root-generated subalgebras arising from $\pi$-systems containing positive roots, determine their sets of real roots, and prove that Dynkin’s bijections hold in this context. Finally, we discuss the graded embedding problem in the super setting and highlight some applications of these embeddings in mathematical physics.

\textit{Acknowledgment: The authors thank R. Venkatesh for several inspiring and helpful discussions. } 
\section{Preliminaries and regular subalgebras}\label{prel}
\subsection{}
We denote by $\mathbb{N}, \mathbb{Z}_{+}$, and $\mathbb{Z}$ the set of positive integers, non-negative integers, and integers, respectively. Unless otherwise stated, all vector spaces are assumed to be complex vector spaces and a super vector space is a $\mathbb{Z}_2$-graded vector space. We denote by $p(x)$ the parity of a homogeneous element $x$ in a super vector space.
\subsection{} Let $I=\{1,\dots,n\}$ be a finite index set, $p:I\rightarrow \bz_{2}$ be a parity function, and $A=(a_{ij}),\ i,j\in I$ be a complex matrix. Fix an even vector space $\lie{h}$ of dimension $2n-\mathrm{rank}(A)$ and linearly independent vectors $\alpha_i\in\lie{h}^*$ and $h_i\in \lie{h}$ such that $\alpha_j(h_i)=a_{ij}$ for all $i,j\in I$; see \cite{Kac90Infinite} for the existence and uniqueness up to a linear transformation of $\lie h$. Define a Lie superalgebra $\tilde{\lie{g}}(A)$ by generators $x_{i}^{\pm},\ i\in I$ and $\lie h$ with relations
$$[h,h']=0,\ \ [h,x_{i}^{\pm}]=\pm \alpha_i(h)x_i^\pm,\ \ [x_i^+,x_j^-]=\delta_{ij}h_i,\ \ p(x_i^\pm)=p(i)$$ for all $h,h'\in \lie{h}$ and $i,j\in I$. 
It is known that there exists a unique maximal ideal $\boldsymbol{\tau}$ of $\tilde{\lie{g}}(A)$ which intersects $\lie{h}$ trivially (see \cite[Proposition 2.2]{vander89Classification}). We define $\lie g=\lie{g}(A)$ to be the quotient of $\tilde{\lie{g}}(A)$ by the ideal $\boldsymbol{\tau}$ and omit the dependence on $A$ when the context is clear. Moreover, we have $\lie{g}(A)\cong \lie{g}(DA)$ for any invertible diagonal matrix $D$ and thus we will assume without loss of generality that $a_{ii}\in\{0,2\}$ for all $i\in I$. Such matrices are called \textit{normalized}.

The action of $\lie h$ on $\lie{g}=\lie g_0\oplus \lie g_1$ is diagonalizable and induces a root space decomposition 
$$\lie g=\lie h \oplus \bigoplus_{\alpha\in \Delta} \lie g_{\alpha},\ \ \ \lie g_{\alpha}:=\{x\in \lie g: [h,x]=\alpha(h)x,\ \forall h\in \lie h\}$$
where $\Delta=\{\alpha\in \lie h^{*}: \lie g_{\alpha}\neq 0 \}$ is called the set of \textit{roots}. We define as usual the set of \textit{even roots} and \textit{odd roots} by 
$$\Delta_0=\{\alpha\in \Delta: \lie g_{\alpha}\subseteq \lie g_0\},\ \ \  \Delta_1=\{\alpha\in \Delta: \lie g_{\alpha}\subseteq \lie g_1\}.$$
Moreover, we denote by $\Delta^{\pm}$ the sets of \textit{positive} and \textit{negative roots}, respectively - that is, those roots which are linear combinations of the $\alpha_i$'s with non-negative and non-positive coefficients, respectively. Then we have
$$\Delta=\Delta_0\cup \Delta_1=\Delta^+\cup \Delta^-$$
and we can extend the parity function $p$ to $\Delta$ in the obvious way. 
The set of simple roots is denoted by $\Pi=\{\alpha_i:i\in I\}$. The Chevalley involution $\tilde{w}:\lie g\rightarrow \lie g$ is defined as follows
$$\tilde{w}(x_i^{+})=-(-1)^{p(i)}x_i^{-},\ \ \tilde{w}(x_i^{-})=-x_i^{+},\ \ \tilde{w}(h)=-h,\ \ h\in \lie h,\ i\in I.$$
Setting $I_1=\{i\in I: p(i)=1\}$ we have that $\tilde{w}$ has order 4 if $I_1\neq \emptyset$ and order 2 otherwise. The center $\mathfrak{c}$ of $\lie g$ is given by $\mathfrak{c}=\{h\in \mathfrak{h}: \alpha_i(h)=0\ \forall i\in I\}$ (see \cite[Section 2.5]{vander89Classification}) and denote by $\lie g'=[\lie g,\lie g]$ the derived superalgebra. We have $\mathfrak{c}\subseteq \mathfrak{g}'$ and $\mathfrak{g}=\mathfrak{g}'\oplus \lie h''$ where $\lie h=\lie h'\oplus \lie h''$ and $\lie h'$ is the span of the $h_i$'s and $\lie h''$ is a linear complement.
\begin{rem}\label{rank1}
In contrast to the situation for Lie algebras, one obtains several distinct types of rank-one subalgebras.
For each $i$, the elements $x_i^{\pm}$ and $h_i$ generate a subalgebra $\mathfrak{g}(\alpha_i)$, which is isomorphic to:
\begin{enumerate}
\item[$\Diamond$] the three-dimensional Heisenberg algebra $\mathfrak{H}_3$ if $a_{ii}=0$ and $p(i)=0$,
\item[$\Diamond$] $\mathfrak{sl}(1,1)$ if $a_{ii}=0$ and $p(i)=1$,
\item[$\Diamond$] $\mathfrak{sl}_2$ if $a_{ii}=2$ and $p(i)=0$,
\item[$\Diamond$] $\mathfrak{osp}(1,2)$ if $a_{ii}=2$ and $p(i)=1$.
\end{enumerate}
We will use this fact repeatedly throughout the remainder of the paper.
\end{rem}
We call $\alpha_i$ \textit{isotropic} if $a_{ii}=0$ and \textit{non-isotropic} otherwise and $\alpha_i$ is called \textit{regular} if $a_{ij}=0$ implies $a_{ji}=0$ for all $j\in I$; otherwise it is called \textit{singular}. The matrix $A$ is called \textit{regular} if each $\alpha_i$ is regular. The matrix $A$ is called \textit{admissible} if the operators $\mathrm{ad}_{x_i^{\pm}}$ act locally nilpotently on $\lie{g}(A)$ for all $i\in I$ and $\lie g(A)$ is called \textit{quasisimple} if any ideal $\lie{i}\subseteq \lie{g}(A)$ satisfies $\lie{i}\subseteq \lie{h}$ or $\lie{i}+\lie{h}=\lie g(A).$
The following lemma can be found in \cite[Section 2]{Serganova11KacMoody}.
\begin{lemma}\label{admatrix}
  The (normalized) matrix $A$ is admissible if and only if the following conditions are satisfied
    \begin{enumerate}
    \item $a_{ii}=0$ and $p(i)=0$ implies $a_{ij}=0$ for all $j\in I,$
    \item $a_{ii}=2$ implies $a_{ij}\in 2^{p(i)}\bz_{\le 0}$ for all $j\in I, \ j\neq i$  
    \item $a_{ii}=2$ and $a_{ij}=0\implies a_{ji}=0.$
    \end{enumerate}
    \qed
\end{lemma}
We say that a matrix $A$ \textit{symmetrizable} if there exists an invertible diagonal matrix $D$ such that $DA=B$ is a symmetric matrix. It is known that a matrix $A$ is symmetrizable if and only if $\lie g(A)$ admits a non-degenerate even invariant bilinear form $(\cdot,\cdot)$ (see \cite[Proposition 4.2]{vander89Classification}) inducing an isomorphism $\nu: \lie h\rightarrow \lie h^{*}$. 
\subsection{} A linearly independent subset $\Sigma\subseteq \Delta$ is called a \textit{base} of $\lie g$ if for each $\alpha\in\Sigma$ we can find $x_\alpha^{\pm}\in \lie g_{\pm \alpha}$ such that $x_{\alpha}^\pm,\alpha\in\Sigma$ and $\lie h$ generate $\lie g$ and $[x_\alpha^+,x_\beta^-]=0$ for all $\alpha,\beta\in \Sigma$ with $\alpha\neq \beta$.
Setting $h_{\alpha}:=[x_{\alpha}^+,x_{\alpha}^-],$ for $\alpha\in \Sigma$, we have 
$$[h,x_{\alpha}^\pm]=\pm \alpha(h)x_\alpha^\pm,\ \ [x_\alpha^+,x_\beta^-]=\delta_{\alpha \beta} h_\beta,\ \ h\in\lie{h}.$$
For example, $\Pi$ is a base called the \textit{standard base}.  We denote by $\lie g(\alpha)$, for $\alpha\in \Sigma$, the subalgebra generated by $x_{\alpha}^{\pm}$. The set $\{h_\alpha:\alpha\in \Sigma\}$ is linearly independent and $|\Sigma|=n$. 
For a base $\Sigma$ we define its \textit{Cartan matrix} by $A_\Sigma=(a_{\alpha\beta}), \alpha,\beta\in\Sigma$ where $a_{\alpha\beta}=\beta(h_\alpha)$. Then we have an isomorphism $\lie g(A_{\Sigma})\xrightarrow[]{\sim} \lie g(A)$; the existence of a surjective homomorphism follows from the definition of a base and the injectivity follows by proving that both have the same Cartan subalgebra $\lie h$.

Given $\alpha\in\Sigma$ with $a_{\alpha\alpha}=0$ and $p(\alpha)=1$ we define $s_\alpha(\beta)$ for $\beta\in \Sigma$ by
$$s_\alpha(\beta)=\begin{cases}
    (-1)^{\delta_{\alpha,\beta}}\beta & \text{ if } a_{\alpha\beta}=a_{\beta\alpha}=0\\
    \beta+\alpha & \text{ if }a_{\alpha\beta}\neq 0 \text{ or }a_{\beta\alpha}\neq 0
\end{cases}$$
and refer to $s_{\alpha}$ ($\alpha$ satisfying the above conditions) as an \textit{odd reflection} on $\Sigma$. 
Define for such a reflection and $\beta\in \Sigma$:
$$x_{s_{\alpha}(\beta)}^{\pm}=x_\alpha^{\mp}, \ \text{if $s_{\alpha}(\beta)=-\alpha$},\ \ x_{s_{\alpha}(\beta)}^{\pm}=[x_\alpha^{\pm},x_\beta^{\pm}],\ \ \text{if $s_{\alpha}(\beta)=\alpha+\beta$}.$$ Note that $x_{s_{\alpha}(\beta)}^{\pm}\neq 0$ can be derived from $a_{\alpha\beta}\neq 0$ or $a_{\beta\alpha}\neq 0$. Then we have for the coroots
$$h_{s_\alpha(\beta)}=[x_{s_\alpha(\beta)}^{+},x_{s_\alpha(\beta)}^{-}]=\begin{cases}
    h_\beta & \text{ if }a_{\alpha\beta}=a_{\beta\alpha}=0,\\ \vspace{5pt}
   (-1)^{p(\beta)}( a_{\alpha\beta}h_\beta+a_{\beta\alpha}h_\alpha) & \text{ if }a_{\alpha\beta}\neq 0 \text{ or }a_{\beta\alpha}\neq 0
\end{cases}$$
A root $\alpha$ in some base $\Sigma$ is called \textit{regular} with respect to $\Sigma$ if $a_{\alpha\beta}=0\implies a_{\beta\alpha}=0$ for all $\beta\in\Sigma.$ Note that this definition extends the previous one for the standard base and depends on the choice of base. Moreover, a root $\alpha$ is called \textit{isotropic} (resp. \textit{non-isotropic}) if $r\alpha$ belongs to some base $\Sigma$ for $r\in \{1,1/2\}$ and $\alpha(h_{r\alpha})= 0$ (resp. $\alpha(h_{r\alpha})\neq 0$). Since $\dim \lie{g}_{\pm r\alpha}=1,$ the definition is independent of the choice of the base.
Given a base $\Sigma$ and a regular odd isotropic root $\alpha\in \Sigma$, it turns out that $s_\alpha(\Sigma)$ is again a base (see \cite[Lemma 3.1]{Serganova11KacMoody}). In this case we say that $s_\alpha(\Sigma)$ is obtained from $\Sigma$ by an odd reflection.
\begin{lemma}\label{basereg}
   Let $\Sigma$ be a base and $\alpha\in \Sigma$ be odd isotropic and regular with respect to $\Sigma$. Then, $s_\alpha(\Sigma)$ is a base. \qed
\end{lemma}
The following definition is taken from \cite[Section 3]{Serganova11KacMoody}.
\begin{definition}
We call $\mathfrak{g}(A)$ a \emph{regular Kac–Moody superalgebra} if, for every base $\Sigma$ obtained from the standard base $\Pi$ by a sequence of odd reflections, the corresponding matrix $A_{\Sigma}$ is regular and admissible.
\end{definition}
\textit{In the rest of this paper, we assume that $\lie g(A)$ is a quasisimple regular Kac-Moody superalgebra. In particular, $A$ is non-zero, indecomposable and admissible.} 
\subsection{}An even root $\alpha\in\Delta_0$ is called \textit{principal} if there exists a base $\Sigma$ obtained from the standard base $\Pi$ by odd reflections such that $\alpha\in \Sigma$ or $\alpha/2\in\Sigma.$ We denote by $\Delta_{\mathrm{pr}}$ the set of principal roots of $\mathfrak{g}$. This set satisfies the properties of a $\pi$-system, to be defined later. By \cite[Section 9]{Serganova11KacMoody}, the set of principal roots is finite. The main result of this paper provides a description of the real roots of the Lie subsuperalgebra generated by the principal roots. It is clear that principal roots are positive and since $\lie g(A)$ is quasisimple, all principal roots are also non-isotropic. Therefore, each $\alpha\in \Delta_{\mathrm{pr}}$ gives a standard $\mathfrak{sl}_2$ triple $\{x_{\alpha}^{\pm}\in \lie g_{\pm \alpha}, h_{\alpha}\in \lie h\}$ and we define the \textit{even reflection} corresponding to $\alpha$ by
$$s_\alpha:\lie{h}^*\to \lie{h}^*,\ \ \lambda\mapsto \lambda-\lambda(h_\alpha)\alpha.$$ 
The even reflections permute the roots and map bases to bases. We define the \textit{Weyl group} $W$ of $\lie g$ to be the group generated by all even reflections. The group $W$ also acts on $\lie h$ by the dual action given by $$s_\alpha(h)=h-\alpha(h)h_\alpha,\ \ \ h\in \lie{h}.$$ By \cite[Lemma 3.8]{Kac90Infinite} for $\alpha\in\Delta_{\mathrm{pr}},$ there exists an automorphism $\tau_\alpha$ of $\lie{g}$ such that $$\tau_\alpha|_{\lie{h}}=s_\alpha,\ \ \tau_\alpha(\lie g_\beta)=\lie g_{s_\alpha(\beta)},\ \ \beta\in\Delta.$$
For an odd reflection $s_{\alpha}$ with $\alpha\in \Sigma$ in some base $\Sigma$, we have
\begin{equation}\label{11a}ws_{\alpha}(\Sigma)=s_{w(\alpha)}w(\Sigma),\ \ w\in W\end{equation}
Moreover, for any even or odd reflection $s_{\alpha}$ with $\alpha\in \Sigma$ in some base $\Sigma$, we have
$$\Delta^+(\Sigma)\backslash\{\alpha,2\alpha\}=\Delta^+(s_{\alpha}(\Sigma))\backslash\{-\alpha,-2\alpha\}$$

 A root $\alpha$ is called \textit{real} if there exists a base $\Sigma$ obtained from the standard base $\Pi$ by using even and odd reflections such that $\alpha\in \Sigma$ or $\alpha/2\in\Sigma.$ A root which is not real is called \textit{imaginary}. We denote by $\Delta^{\mathrm{re}}$ (resp. $\Delta^{\mathrm{im}}$) the set of real (resp. imaginary) roots of $\lie g.$  It is also not hard to show with \eqref{11a} that the Weyl group is also generated by all $s_{\alpha}$ with $\alpha\in \Delta^{\mathrm{re}}$ and non-isotropic. 
 \begin{rem}\label{triple}We record the following facts for later use.
\begin{enumerate}
\item We will frequently use \cite[Lemma 4.11]{Serganova11KacMoody} which we summarize below
\begin{itemize}
    \item If $\alpha$ is real odd isotropic root and $k\alpha\in\Delta$, then $k\in\{\pm 1\}$.
    \item If $\alpha$ is real odd non-isotropic root and $k\alpha\in\Delta$, then $k\in\{\pm 1, \pm 2\}$.
    \item If $\alpha$ is real even root and $k\alpha\in\Delta$, then $k\in\{\pm 1, \pm \frac{1}{2}\}$.
\end{itemize}
    \item As for roots in a base, we can choose for any $\alpha\in \Delta^{\mathrm{re}}$ suitable $x_{\alpha}^{\pm}\in \lie g_{\pm \alpha}$ such that $\alpha(h_{\alpha})\in\{0,2\}$ where $h_{\alpha}=[x_{\alpha}^{+},x_{\alpha}^{-}]$. If $\alpha\in \Sigma$ (recall the definition of real roots) then this follows from the definition of a base after a suitable scaling and if $\alpha/2\in\Sigma$, we must have that $\alpha$ is an even non-isotropic real root and hence we are done with Remark~\ref{rank1}. \textit{We will make this choice for the rest of the paper.} In particular, for any real odd non-isotropic root $\alpha$, we must have $h_{2\alpha}=\frac{1}{2}h_\alpha$. If $\lie g$ is symmetrizable we have $\nu(h_{\alpha})=\frac{2\alpha}{(\alpha,\alpha)}$ for all real non-isotropic roots $\alpha$ and for all real isotropic roots $\alpha$ we have that $\nu(h_{\alpha})$ is proportional to $\alpha$.
    \item Let $\alpha$ be a real root, then $-\alpha$ is also real. This follows from Lemma~\ref{basereg} and the regularity of $\lie g$.
\end{enumerate}
 \end{rem}
  We now discuss examples of Kac-Moody superalgebras that are closely related to Kac-Moody algebras.
 \begin{example}\label{examnoniso}
     Assume that $\mathfrak g$ has no isotropic real roots. Then we have $$\Delta_{\mathrm{pr}}=\{\alpha_i:i\in I\backslash I_1\}\cup\{2\alpha_i:i\in I_1\}.$$ 
     Let $B$ be the generalized Cartan matrix associated to $\Delta_{\mathrm{pr}}$. The Weyl group $W$ of $\mathfrak g$ is isomorphic to the Weyl group of the Kac–Moody algebra $\mathfrak g(B)$. In particular, $W$ is a Coxeter group and
     $\Delta^{\mathrm{re}}=W(\Pi\cup\{2\alpha_i:i\in I_1\}).$ We shall freely use the Coxeter structure of $W$.
 \end{example}
We collect a useful property of imaginary roots in the rest of this subsection. For a base $\Sigma$ consider the convex cone $C_\Sigma^+:=\bz_{+}\Sigma$ and let 
 $$C(\Pi)=\bigcap_{\Sigma}C_\Sigma^+,\ \ \ \  D(\Pi)=\bigcap_{\Sigma}C_\Sigma^+$$ where the first (resp. second) intersection runs over all bases $\Sigma$ obtained from $\Pi$ by applying odd (resp. even and odd) reflections. We obviously have $\Delta_{\mathrm{pr}}\subseteq C(\Pi)$ and  
 from \eqref{11a} we also obtain $$D(\Pi)=\bigcap_{w\in W}wC(\Pi).$$
 The first part of the following result is proven in \cite[Corollary 4.12]{Serganova11KacMoody} and the proof of the second part is straightforward. We give the details for completeness.
 \begin{lemma}\label{lemimroots}
 \begin{enumerate}
     \item If $\alpha\in \Delta^{\mathrm{im}},$ then either $\alpha$ or $-\alpha$ belongs to $D(\Pi).$
     \item Every even real root is non-isotropic.
 \end{enumerate}
    \begin{proof}
    Let $\alpha$ be an even real root. Then there exists a base $\Sigma$, obtained from the standard base by using even and odd reflections, such that either $\alpha/2$ or $\alpha$ is in $\Sigma.$ 
    In the former case, the root $\alpha/2$ is odd non-isotropic (see Remark~\ref{rank1}) and therefore $\alpha$ is non-isotropic. In the latter case, since $\lie g(A)\cong \lie g(A_{\Sigma})$ is quasisimple and $A_{\Sigma}$ is admissible, we see with \cite[Lemma 2.3]{Serganova11KacMoody} that $\alpha(h_\alpha)=0$ is only possible if $\alpha$ is odd.
\end{proof}
 \end{lemma}
\subsection{}\label{secbasicdef}
We end this section with a few more definitions. Given a subset $X\subseteq \Delta^{\mathrm{re}}$ and an isotropic root $\alpha\in X$ (which is automatically odd by Lemma~\ref{lemimroots}), we define for $\beta\in X$
\begin{equation}\label{defoddref}s_\alpha(\beta)=\begin{cases}
    -\beta &\ \ \text{ if }\beta=\pm\alpha,\\
    \beta+\alpha &\ \ \text{ if } \beta\neq \pm\alpha \text{ and } \alpha+\beta \text{ is a real root,}\\
    \beta &\ \ \text{ otherwise }
\end{cases}\end{equation}
and call $s_{\alpha}$ the \textit{odd reflection} on $X$.
\begin{rem}
    Note that the above definition extends the definition of odd reflections introduced before when we have the situation $X=\Sigma\subseteq \Delta^{\mathrm{re}}$ for a base $\Sigma$ and $\alpha\in\Sigma$ being a regular (with respect to $\Sigma$) isotropic root.  This can be seen from the following  observations. Let $\beta\in \Sigma$. First of all we have 
    $$a_{\alpha\beta}\neq 0 \iff \alpha+\beta \text{ is a real root }\iff \alpha+\beta \text{ is a root}$$
    where the forward directions are straightforward. Now, if $\alpha+\beta$ is a root and $a_{\alpha\beta}=0$, we would also have  $a_{\beta\alpha}=0$ since $\alpha$ is regular. Hence, the ideal generated by $[x^+_\alpha,x^+_\beta]$ is a non-zero proper ideal intersecting the Cartan subalgebra trivially. This is a contradiction.
\end{rem}
We recall the definition of regular subalgebras and some combinatorial subsets associated to them.
\begin{definition}\begin{enumerate}
    \item  A $\lie h$-invariant Lie subsuperalgebra of $\lie{g}$ is called a \textit{regular subalgebra}. Given a regular subalgebra $\lie{s}$, we define
$$\Delta(\lie{s}):=\{\alpha\in\Delta:(\lie{g}_{\alpha}\cap \lie s)\neq 0\},\ \ \lie{s}_{\alpha}:=\lie{g}_\alpha\cap \lie{s},\ \alpha\in\Delta(\lie s)$$ to be the set of roots of $\lie s$ with respect to $\lie h$. Furthermore, we set 
$$\Delta(\lie{s})^{\mathrm{re}}=\Delta^{\mathrm{re}}\cap \Delta(\lie{s}),\ \ \Delta(\lie{s})^{\mathrm{im}}=\Delta^{\mathrm{im}}\cap \Delta(\lie{s}),$$ 
$$\Delta(\lie{s})^{\mathrm{re},+}=\Delta(\lie{s})^{\mathrm{re}}\cap \Delta^+,\ \ \Delta(\lie{s})^{\mathrm{im},+}=\Delta(\lie{s})^{\mathrm{im}}\cap \Delta^+.$$
\item A regular subalgebra $\lie s$ is called a \textit{root generated subalgebra} if there exists a subset $S\subseteq \Delta^{\mathrm{re}}$ such that $\lie s=\lie g(S)$, where $\lie g(S)$ is defined to be the Lie subsuperalgebra generated by all root spaces $\lie g_{\alpha}$, $\alpha\in \pm S$.
\end{enumerate}
\end{definition} 
Note that a regular subalgebra admits the following decomposition
$$\lie{s}=(\lie h\cap \lie s)\oplus \bigoplus_{\alpha\in \Delta(\lie s)}\lie{s}_\alpha.$$ 

A non-empty subset $\Psi\subseteq \Delta^{\mathrm{re}}$ is called 
\begin{itemize}
\item \textit{ symmetric} if $\Psi=-\Psi$.\vspace{0,2cm}

    \item \textit{ a subroot system} if $s_{\alpha}(\beta)\in \Psi$ for all $\alpha,\beta\in \Psi$.\vspace{0,2cm}
       
    \item \textit{ real closed} if $\alpha,\beta\in \Psi$ and $\alpha+\beta\in \Delta^{\mathrm{re}}$ implies $\alpha+\beta\in \Psi.$\vspace{0,2cm}
\end{itemize}
For simplicity, we will write closed instead of real closed from now on. Note that every subroot system is symmetric, and by Remark~\ref{triple}, the set $\Delta^{\mathrm{re}}$ is a closed subroot system.

\subsection{} The following lemma can be found in \cite[Lemma 1.94]{Wa01}. 

\begin{lemma}\label{osp}
   The finite-dimensional irreducible $\lie{osp}(1,2)$-modules are up to isomorphism in bijective correspondence with $\bz_+$. To be more precises, for each $k\in\bz_+,$ the associated irreducible module has a basis $\{v_j: 0\le j\le 2k\}$ with action 
    $$v_j=f^jv_0,\ \ \ hv_j=(2k-2j)v_j$$
    $$ev_{j}=\begin{cases}
        -jv_{j-1} &\text{ if $j$ is even},\\
        (2k+1-j)v_{j-1} &\text{ if $j$ is odd}
    \end{cases}$$
   where $\{e,f\}$ is a basis of $\lie{osp}(1,2)_1$ and $h=[e,f].$
\qed
\end{lemma}
The next proposition proves important properties of $\Delta(\lie s)^{\mathrm{re}}$.

\begin{prop}\label{basiclem}
Let $\lie{s}$ be a regular subalgebra of $\lie{g}$ and $\alpha,\beta\in\Delta(\lie{s})^{\mathrm{re}}.$ 
\begin{enumerate}
    \item If $\alpha$ or $\beta$ is non-isotropic, then $[\lie{s}_\alpha,\lie{s}_\beta]\neq 0$ if and only if $\alpha+\beta\in\Delta.$ \vspace{0,1cm}
    \item If both $\alpha$ and $\beta$ are isotropic such that $(\alpha+\beta)\in\Delta^{\mathrm{re}}$ and
    \begin{equation}\label{isocond12}h_{\alpha+\beta}\in \mathrm{span}_{\mathbb{C}}\{h_{\alpha},h_{\beta}\},\end{equation} then $[\lie{s}_\alpha,\lie{s}_\beta]\neq 0.$\vspace{0,1cm}
    \item If $\alpha$ is non-isotropic such that $-\alpha\in \Delta(\lie{s})$, then we have $s_\alpha(\beta)\in\Delta(\lie s)^{\mathrm{re}}.$ 
\end{enumerate}
\end{prop}

\begin{proof}
    We first prove part $(1).$ Without loss of generality we assume that $\alpha$ is a non-isotropic real root. Hence $\mathrm{ad}_{x_{\alpha}^{\pm}}$ acts locally nilpotently and $\lie g(\alpha)$ is isomorphic to $\lie{sl}_2$ or $\lie{osp}(1|2)$ by Remark~\ref{rank1}. If $\alpha+\beta$ is a root and $[\lie{s}_\alpha,\lie{s}_\beta]=0$, we consider the  $\lie g(\alpha)$-module generated by $x_{\beta}^+$ which is finite-dimensional and irreducible.  Thus $\beta(h_{\alpha})\geq 0$ by the representation theory of $\lie{sl}_2$ or $\lie{osp}(1,2)$ (see Lemma~\ref{osp}). By assumption, there exists a non-zero element $v\in \lie g_{\alpha+\beta}$ and the $h_{\alpha}$-eigenvalue is strictly positive. Thus $0\neq [x_{\alpha}^{+},[x_{\alpha}^{-},v]]\in [\lie{s}_\alpha,\lie{s}_\beta]$, which is a contradiction.
    
    To prove $(2)$ assume by contradiction $[x^+_\alpha,x^+_\beta]=0.$ Note that $\alpha,\beta$ are real isotropic odd roots and $\alpha+\beta$ is an even non-isotropic real root by Lemma~\ref{lemimroots}. Hence $\lie{g}(\alpha+\beta)$ is isomorphic to $\lie{sl}_2$ by Remark~\ref{rank1}. Moreover, $(\alpha+\beta)-\alpha=\beta$ implies that $[x^+_{\alpha+\beta},x^+_{-\alpha}]=-ax^+_\beta$ for some non-zero scalar $a$ by part (1) applied to $\lie s=\lie g$. Now, using the super Jacobi identity we obtain (note that $[x^+_{-\alpha},x^+_{-\alpha}]=0$, since $\alpha$ is odd isotropic)
    $$[x^+_{-\alpha},[x^+_{-\alpha},x^+_{\alpha+\beta}]]= -[x^+_{-\alpha},[x^+_{-\alpha},x^+_{\alpha+\beta}]]\implies [x^+_{-\alpha},[x^+_{-\alpha},x^+_{\alpha+\beta}]]=0.$$
    Thus $a[x^+_{-\alpha},x^+_\beta]=0$ and so $[x^+_{-\alpha},x^+_{\beta}]=0.$ Consequently, the vector $x^+_\beta$ is killed by both $x^+_\alpha$ and $x^+_{-\alpha}$ and so $[h_\alpha,x^+_\beta]=0.$ This implies $\beta(h_\alpha)=0$ and similarly one can show $\alpha(h_\beta)=0.$ This is absurd by \eqref{isocond12}, since $\alpha+\beta$ is non-isotropic. 
    
The proof of $(3)$ is similar to \cite[Proposition 2.4]{idv2023root} and we give some details. If $\beta(h_\alpha)=0,$ then $s_\alpha(\beta)=\beta$
        and there is nothing to show. So assume that $\beta(h_\alpha)\neq 0.$ Since the $\lie g(\alpha)$ module generated by $x_\beta^+$ is finite dimensional, it follows from the representation theory of $\lie{sl}_2$ or $\lie{osp}(1,2)$ that 
        \begin{align*}
            (\mathrm{ad}\, x_\alpha^-)^{\beta(h_\alpha)}(x_{\beta}^+)=[x_\alpha^-,[x_\alpha^-,\dots [x_\alpha^-,[x_\alpha^-, x_{\beta}^+]]\cdots ]\neq 0 &\ \text{ if }\beta(h_\alpha)>0,\\
            (\mathrm{ad}\, x_\alpha^+)^{-\beta(h_\alpha)}(x_{\beta}^+)=[x_\alpha^+,[x_\alpha^+,\dots [x_\alpha^+,[x_\alpha^+, x_{\beta}^+]]\cdots ]\neq 0 &\ \text{ if }\beta(h_\alpha)<0,
        \end{align*}
        Therefore, we have $s_\alpha(\beta)=\beta-\beta(h_\alpha)\alpha\in \Delta(\lie s)^{\mathrm{re}}$.
\end{proof}
We record the following remark which follows from \cite[Theorem 2.27]{hoyt10regular} and \cite[Theorem 5.4]{vander89Classification} (see also \cite[Theorem 7.3]{Serganova11KacMoody}). 
\begin{rem}\label{remafftw}
If $\lie g$ is additionally symmetrizable and $\Delta^{\mathrm{re}}$ contains at least one isotropic root (equivalently, one isotropic simple root), then $\lie g$ is isomorphic to a finite-dimensional, affine or twisted affine Kac-Moody superalgebra. In this case we have an explicit description of root systems; see for example \cite[Section 2.5.4]{kac77liesuper} for finite types and \cite[Table 5]{abbas21affine} for (twisted) affine types. These descriptions will be used throughout the article, whenever necessary, without further reference.
\end{rem}
\begin{cor}\label{finaffine}
Let $\lie{g}$ be symmetrizable and $\lie{s}$ be a regular subalgebra. 
\begin{enumerate}
    \item If $\alpha,\beta\in \Delta(\lie s)^{\mathrm{re}},$ then $[\lie s_\alpha,\lie s_\beta]\neq 0$ if and only if $\alpha+\beta\in \Delta.$
    \item If $\Delta(\lie{s})^{\mathrm{re}}$ is symmetric, then $\Delta(\lie{s})^{\mathrm{re}}$ is a closed subroot system.
    \item Let $\alpha,\beta\in \Delta^{\mathrm{re}}$ be such that $\alpha$ is isotropic. Then $\alpha\pm\beta\in \Delta^{\mathrm{re}}$ implies $\alpha\mp\beta\notin \Delta$.
\end{enumerate}
\begin{proof}
Since $\lie{g}$ is symmetrizable, property \eqref{isocond12} always holds and part (2) follows from Proposition~\ref{basiclem}. Part (1) again follows from Proposition~\ref{basiclem} unless $\alpha,\beta$ are isotropic and $\alpha+\beta$ is imaginary. In this case the claim follows from Remark~\ref{remafftw} by the explicit description of roots. For part $(3),$ assume that $\alpha\pm\beta$ is real. Then we have by part (1) (up to non-zero scalars)
$$x_{\pm \beta}^+=[x_\alpha^-,x_{\alpha\pm\beta}^+] \implies [x_\alpha^-,x_{\pm \beta}^+]=[x_\alpha^-,[x_\alpha^-,x_{\alpha\pm\beta}^+]]=0\implies \alpha\mp\beta\notin \Delta.$$ 
\end{proof}
\end{cor}
We end this section with the following example of a subroot system generated by two isotropic real roots.
\begin{example}
    Let $\alpha$ and $\beta$ be two isotropic real roots satisfying \eqref{isocond12} such that $\alpha+\beta$ is a real root. In particular, $\alpha+\beta$ is an even real root and thus non-isotropic. Then we have $[x^+_\alpha,x^+_\beta]\neq 0$ and $[x^-_{\alpha+\beta},x^+_\beta]\neq 0$ by Proposition~\ref{basiclem}. Moreover, $[[x^+_\alpha,x^+_\beta],x^+_\beta]=0$ and hence by considering the $\lie g(\alpha+\beta)$ (which is isomorphic to $\mathfrak{sl}_2$) module generated by $x^+_\beta$ we get $\beta(h_{\alpha+\beta})\in \mathbb{N}$. Similarly we obtain in the same way $\alpha(h_{\alpha+\beta})\in \mathbb{N}$ and thus 
$\alpha(h_{\alpha+\beta})=\beta(h_{\alpha+\beta})=1$ since $(\alpha+\beta)(h_{\alpha+\beta})=2$. This implies $\{\pm \alpha,\pm \beta, \pm (\alpha+\beta)\}$ is a subroot system, since
$$s_{\alpha+\beta}(\alpha)=-\beta,\ \ s_{\alpha+\beta}(\beta)=-\alpha,\ \ $$
\end{example} 

\section{Root strings of Kac-Moody superalgebras}\label{secrootstring}
In this section, we study the root strings of Kac-Moody superalgebras. The Lie algebra counterpart of this section can be found in \cite{billig1995root,morita1988root}. For $\alpha\in\Delta^{\mathrm{re}}$ and $\beta\in\Delta$, the $\alpha$-string through $\beta,$ denoted as $\mathbb{S}(\beta,\alpha),$ is defined by $$\mathbb{S}(\beta,\alpha):=\{\beta+k\alpha:k\in\mathbb{Z}\}\cap\Delta.$$
Note that, if $\alpha$ and $\beta$ are scalar multiples of each other we have
\begin{equation}\label{scalarmultiple}\mathbb{S}(\beta,\alpha)\subseteq \{\pm \alpha, \pm 2\alpha\},\ \ \text{if}\ \ \beta\in\{\pm\alpha,\pm 2\alpha\}, \ \ \mathbb{S}(\beta,\alpha)\subseteq \left\{\pm \frac{\alpha}{2}\right\},\ \ \text{if} \ \ \beta\in\left\{\pm\frac{\alpha}{2}\right\}\end{equation}
\subsection{} The proof of the first part of the following proposition is analogous to \cite[Proposition 5.1]{Kac90Infinite}. We give the details for completeness.

\begin{prop}\label{unbrokenrs}
Let $\alpha,\beta\in \Delta$ be two roots of a Kac-Moody superalgebra such that $\alpha$ is real.

\begin{enumerate}
    \item  If $\alpha$ is non-isotropic, then there exist non-negative integers $p$ and $q$ related by the equation $p-q=\beta(h_\alpha)$ such that $\beta+k\alpha\in \Delta\cup \{0\}$ if and only if $-p\le k\le q.$ In particular, $s_\alpha$ reverses the root string $\mathbb{S}(\beta,\alpha)$. \vspace{0,1cm}
    \item If $\alpha$ is isotropic, then $|\mathbb{S}(\beta,\alpha)|$ is finite.
\end{enumerate}
\end{prop}
\begin{proof}
 We start proving part (1). The result is easily checked if $\beta$ and $\alpha$ are scalar multiple of each other. Otherwise, consider the subspace $U=\sum_{k\in \mathbb{Z}} \lie g_{\beta+k\alpha}$ of $\lie g.$ Since $\alpha$ is non-isotropic, the subalgebra $\lie g(\alpha)$ 
    is isomorphic to either $\lie{sl}_2$ or $\lie{osp}(1,2).$ Let $v$ be a root vector in $U.$ Since $\mathrm{ad}_{x_\alpha^\pm}$ acts locally nilpotently, the $\lie g{(\alpha)}$ submodule of $U$ generated by $v$ is finite-dimensional. Thus, by the representation theory of $\lie{sl}_2$ or $\lie{osp}(1,2)$, we have that $U$ is a sum of finite-dimensional irreducible $\lie g{(\alpha)}$-modules. 
    Since the eigenvalues of $h_\alpha$ in $U$ are of the form $\beta(h_\alpha)+2k$, it follows that every irreducible module that occur in $U$ contains a non-zero vector in $\lie g_{\beta+t\alpha}$ where $$t=\begin{cases}
        \frac{-\beta(h_\alpha)}{2}& \text{ if } \beta(h_\alpha) \text{ is even},\\
        \frac{1-\beta(h_\alpha)}{2}& \text{ if } \beta(h_\alpha) \text{ is odd}.
    \end{cases}$$
    Therefore, $U$ is equal to the $\lie g{(\alpha)}$ submodule generated by the root space $\lie g_{\beta+t\alpha}$. In particular, $U$ is finite-dimensional and the claim follows from the representation theory of $\lie g(\alpha).$
    
  Now we prove part (2). If $\alpha$ and $\beta$ are scalar multiples of each other, the statement is clear. Since $\alpha$ is real and odd we must have $\alpha\in\Sigma$ for a base $\Sigma$. Moreover, 
\begin{equation}\label{dreir}[x^+_{\alpha},[x_{\alpha}^+,y]]=0,\ \ y\in \lie g.\end{equation} We first assume that $\beta\in C^+_\Sigma$, i.e., $\beta=\sum_{\gamma\in \Sigma}a_{\gamma}\gamma$ where $a_{\gamma}\geq 0$. If $k\in \mathbb{Z}_+$ is such that $\beta+k\alpha\in \Delta$, we can express any element in $\mathfrak{g}_{\beta+k\alpha}$ as a linear combination of Lie words in the root vectors corresponding to the base $\Sigma$. However, by \eqref{dreir} we must have 
$$\mathfrak{g}_{\beta+k\alpha}\neq 0 \implies k\leq \sum_{\gamma\in \Sigma\setminus\{\alpha\}}a_\gamma-a_\alpha.$$
Moreover, if $t > a_\alpha$, we claim that $\beta - t\alpha$ is never a root. Indeed, we have
$$\beta - t\alpha = \sum_{\gamma \in \Sigma \setminus \{\alpha\}} a_{\gamma} \gamma + (a_\alpha - t)\alpha.$$
Since $\beta$ is not a scalar multiple of $\alpha$, there exists some $\gamma \in \Sigma \setminus \{\alpha\}$ with $a_{\gamma} > 0$. Since $t > a_\alpha$, we must have $\beta - t\alpha\notin \Delta$. The proof  for the case $\beta\in -C^+_\Sigma$ is similar.
\end{proof}

We will identify the real roots in a root string. We begin with the following lemma. 
\begin{lemma}\label{lemcontimroots}
    Let $\alpha,\beta\in \Delta$ be such that $\alpha$ is real. Then any root lying in between two imaginary roots in $\mathbb{S}(\beta,\alpha)$ is imaginary.
\end{lemma}
\begin{proof}
Assume that there exist a real root $\beta'\in \mathbb{S}(\beta,\alpha)$ and $r,k\in\mathbb{N}$ such that $\beta'-r\alpha,\beta'+k\alpha\in \Delta^{\mathrm{im}}.$ Note that $\beta$ and $\alpha$ are not scalar multiple of each other by \eqref{scalarmultiple}.
    Since $\alpha$ and $\beta'$ are both real roots, there exist bases $\Sigma'$ and $\Sigma''$ and $s,t\in \{1,1/2\}$ such that $s\beta'\in \Sigma'$ and $t\alpha\in\Sigma''.$ Using Lemma \ref{lemimroots}, we first assume that $\beta'-r\alpha\in D(\Pi)$. Note that $t\alpha\in\Sigma''$ gives $s\beta'\in \mathbb{Q}_{\geq 0}\Sigma''$ and since $s\beta'$ is a root we must actually have $s\beta'\in C_{\Sigma''}^+$. Moreover, we write $\beta'-r\alpha=\sum_{\gamma\in\Sigma'}c_\gamma\gamma$ for some non-negative integers $c_\gamma.$ Now since
    $$r\alpha=(-c_{s\beta'}+1/s)s\beta'+\sum_{\gamma\in\Sigma'\backslash\{s\beta'\}}(-c_\gamma)\gamma$$
    and  $\alpha$ and $\beta'$ are not scalar multiples of each other, it follows that there exists $\gamma\in\Sigma'\backslash\{s\beta'\}$ such that $c_\gamma>0.$ This gives $\beta'+k\alpha\in -D(\Pi)$ since $\beta'+k\alpha$ is an imaginary root and 
    $$\beta'+k\alpha=\left[\frac{1}{s}+\frac{k}{r}\left(\frac{1}{s}-c_{s\beta'}\right)\right]s\beta'+\frac{k}{r}\left(\sum_{\gamma\in\Sigma'\backslash\{s\beta'\}}(-c_\gamma)\gamma\right)\in -C_{\Sigma'}^+  .$$
    We write $\beta'+k\alpha=\sum_{\mu\in \Sigma''}d_\mu\mu$ with $d_\mu\in \bz_{\le 0}$ and at least one of the coefficients is strictly negative. Then we have $$\beta'=(d_{t\alpha}-k/t)t\alpha+\sum_{\mu\in \Sigma''\backslash\{t\alpha\}} d_\mu\mu\in -C^+_{\Sigma''},$$ which contradicts the fact that $s\beta'\in C_{\Sigma''}^+.$ The proof for the case  $\beta'-r\alpha\in -D(\Pi)$ is similar.
\end{proof}
The next corollary characterizes the shape of $\mathbb{S}(\beta,\alpha)$ when it contains a real root. A black circle denotes a real root, while a white circle indicates an imaginary root.
\begin{cor}\label{proprootstring}
    Let $\alpha,\beta\in \Delta$ be such that $\alpha$ is real and non-isotropic. If $|\mathbb{S}(\beta,\alpha)\cap\Delta^{\mathrm{re}}|\geq 1,$ then $\mathbb{S}(\beta,\alpha)$ is of the following form:
    $$\overbrace{\tikzcircle[fill=black]{3.5pt}\ \ \ \tikzcircle[fill=black]{3.5pt} \ \ \cdots\ \  \tikzcircle[fill=black]{3.5pt}}^{p-\text{times}}\ \ \ \overbrace{\tikzcircle{3.5pt}\ \ \ \tikzcircle{3.5pt}\ \  \cdots \ \ \tikzcircle{3.5pt}}^{q-\text{times}} \ \ \ \ \overbrace{\tikzcircle[fill=black]{3.5pt}\ \ \ \tikzcircle[fill=black]{3.5pt}\ \ \cdots \ \ \tikzcircle[fill=black]{3.5pt}}^{r-\text{times}}$$
    Moreover, if $q\neq 0,$ then $p=r.$
\end{cor}
\begin{proof}
The proof follows from Proposition~\ref{unbrokenrs}, Lemma~\ref{lemcontimroots} (left most and right most roots in the string have to be real) and the fact that $s_\alpha$ ($\alpha$ non-isotropic) reverses the root string $\mathbb{S}(\beta,\alpha)$ and maps real to real roots. 
\end{proof}  
\subsection{} \label{subsection32}
In this subsection, we shall prove the results analogous to Morita in \cite{morita1988root} and Billig--Pianzola established in \cite{billig1995root}. 
\begin{lemma}\label{sumnotreal}
    Let $\alpha$ and $\beta$ be two real non-isotropic roots satisfying $\alpha(h_\beta)<-1$ and $\beta(h_\alpha)<-1.$ Then $\alpha+\beta$ is not a real root.
\end{lemma}
\begin{proof}
Without loss of generality we can assume that $\beta$ and $\alpha$ are not scalar multiples of each other. If possible, assume that $\alpha+\beta$ is a real root. First we consider the case when $\alpha+\beta$ is an even root. Note that we have  $2(\alpha+\beta)\notin \Delta$ by Remark~\ref{triple}(1). If $\alpha$ is odd (then $\beta$ is also odd), then $2\alpha\in \Delta.$ Since $s_\beta(2\alpha)=2\alpha-2\alpha(h_\beta)\beta\in \Delta$ and $\alpha(h_\beta)<-1,$ it follows that the root string $\mathbb{S}(2\alpha,\beta)$ is broken, which is a contradiction. If $\alpha$ is even, then $\beta$ is also even and therefore $2\alpha\notin \Delta$ and $2\beta\notin \Delta.$ Since $s_\beta(\alpha)=\alpha-\alpha(h_\beta)\beta\in \Delta$ and $\mathbb{S}(\alpha,\beta)$ is unbroken, we have $\alpha+2\beta\in \Delta$. Hence $\mathbb{S}(\alpha+2\beta,\alpha)=\{\alpha+2\beta\}$ and $(\alpha+2\beta)(h_\alpha)=0$ again by Proposition~\ref{unbrokenrs}(1). Thus $\beta(h_\alpha)=-1$ which is a contradiction. 

Now assume that $\alpha+\beta$ is an odd real root and suppose without loss of generality that $\alpha$ is odd and $\beta$ is even. If $\alpha+\beta$ is isotropic, then $2(\alpha+\beta)$ is not a root (again by Remark~\ref{triple}(1)). Now $s_\beta(2\alpha)\in\mathbb{S}(2\alpha,\beta)$ and $\alpha(h_\beta)<-1$ together imply that the root string $\mathbb{S}(2\alpha,\beta)$ is broken, which is a contradiction. Lastly, assume that $\alpha+\beta$ is odd non-isotropic. Then $2(\alpha+\beta)$ is a real root. Note that the roots $\beta$ and $2\alpha$ are both even and real and $2\alpha(h_\beta)\leq -4<-1$ holds. If $\beta(h_{2\alpha})<-1$, then by the proof of the first case, the root $\beta+2\alpha$ is imaginary (again the fact that $\mathbb{S}(\beta,2\alpha)$ is unbroken guarantees that $\beta+2\alpha$ is a root). Since $s_\beta(2\alpha+\beta)=2\alpha+t\beta$ for some $t\geq 3,$ we obtain that in the root string $\mathbb{S}(2\alpha,\beta)$ in between two imaginary roots $2\alpha+\beta$ and $2\alpha+t\beta,$ there is a real root namely $2(\alpha+\beta),$ which contradicts Lemma \ref{lemcontimroots}. If $\beta(h_{2\alpha})=-1$, then $s_\alpha(\beta)=\beta+2\alpha$ is an even real root and thus $2\beta+4\alpha\notin \Delta.$ Since $\alpha(h_\beta)<-1$, we have $\alpha+2\beta\in\Delta$ and therefore $s_\alpha(\alpha+2\beta)=3\alpha+2\beta\in\Delta.$ Now using Proposition \ref{unbrokenrs}, we obtain that $\mathbb{S}(3\alpha+2\beta,\beta)\subseteq \{\beta+3\alpha,2\beta+3\alpha\}$. Consequently, we get $(3\alpha+2\beta)(h_\beta)\in\{0,1\}$ and thus $\alpha(h_\beta)\in \{-1,-\frac{4}{3}\}$ which is a contradiction. This completes the proof.
\end{proof}

We also have an analogous result to \cite[Exercise 5.14]{Kac90Infinite}.

\begin{prop}\label{prop32}
    Let $\alpha$ and $\beta$ be two real roots in $\Delta$ such that $\alpha$ is non-isotropic.
    \begin{enumerate}
        \item If $\beta$ is non-isotropic and $\beta-\alpha\notin \Delta$ , then the root string $\mathbb{S}(\beta,\alpha)$ contains at most four real roots. \vspace{0,1cm}
        \item If $\lie g$ is symmetrizable, then the root string $\mathbb{S}(\beta,\alpha)$ contains at most four real roots.
    \end{enumerate}
    
\end{prop}
\begin{proof} By \eqref{scalarmultiple}, we can suppose that $\alpha$ and $\beta$ are not scalar multiples of each other.
    First we prove part $(1).$ We have $\alpha(h_\beta)\le 0$ and $\beta(h_\alpha)\le 0$ again by using the representation theory of $\mathfrak{sl}_2$ or $\mathfrak{osp}(1,2)$. The result is immediate if $\beta(h_\alpha)\geq -3$. Therefore we may assume that $\beta(h_\alpha)\le -4.$ If $\alpha(h_\beta)<-1,$ the result follows from Lemma~\ref{sumnotreal} ($\alpha+\beta$, $s_{\alpha}(\alpha+\beta)$ are imaginary) and Corollary~\ref{proprootstring}. Thus it remains to consider the case $\beta(h_\alpha)\le -4$ and $\alpha(h_\beta)=-1$. Note that the condition $\alpha(h_\beta)=-1$ implies that $\beta$ is an even root. Using Corollary~\ref{proprootstring}, it suffices to show that $\beta+2\alpha\notin \Delta^{\mathrm{re}}$. So suppose by contradiction that $\beta+2\alpha\in \Delta^{\mathrm{re}}.$ Since $\beta(h_\alpha)\le -4,$ it follows that $\beta+3\alpha\in \Delta$ and $r:=-2\beta(h_\alpha)-3\geq 5$. The following identities
    $$s_{\beta}(\beta+3\alpha)=2\beta+3\alpha,\ \ s_\alpha(2\beta+3\alpha)=2\beta+r\alpha$$
    imply that $2\beta+5\alpha\in \Delta.$ Since $\beta+2\alpha$ is an even real root, we have $2\beta+4\alpha\notin \Delta.$ But this contradicts the fact that the root string $\mathbb{S}(2\beta+3\alpha,\alpha)$ is unbroken.

    We now prove part (2). Let $\beta':=\beta-p\alpha\in \mathbb{S}(\beta,\alpha)$ be the left most root in the string, which is real by Corollary~\ref{proprootstring}. Since $\mathbb{S}(\beta,\alpha)=\mathbb{S}(\beta',\alpha)$, it suffices by part (1) to consider the case where $\beta'$ is an isotropic real root such that $\alpha + \beta' \in \Delta^{\mathrm{re}}$ (otherwise we are done again by Corollary~\ref{proprootstring}). In particular, it will be enough to show that $|\beta'(h_\alpha)| \le 2$.
    If $\alpha+\beta'$ is isotropic, then $(\alpha + \beta', \alpha + \beta') = 0$, which gives:
    $$\beta'(h_\alpha)=\frac{2(\alpha,\beta')}{(\alpha,\alpha)}=-1.$$
    If $\alpha + \beta'$ is non-isotropic, then $\alpha(h_{\alpha + \beta'})$ must lie in $\mathbb{Z}$. Thus
    $$\alpha(h_{\alpha+\beta'})=\frac{2(\alpha+\beta',\alpha)}{(\alpha+\beta',\alpha+\beta')}=\frac{2+\beta'(h_\alpha)}{1+\beta'(h_\alpha)}=1+\frac{1}{1+\beta'(h_\alpha)}.$$
    The expression on the right takes integer values only when $\beta'(h_\alpha) \in\{0,-2\}$. Therefore, in all cases we have $|\beta'(h_\alpha)| \le 2$. This completes the proof.
\end{proof}
A similar computation gives the following remark.
\begin{rem}\label{remintiso} Let $\lie g$ be symmetrizable and  $\alpha,\beta\in\Delta^{\mathrm{re}}$ be such that $\beta$ is isotropic and $\alpha$ is non-isotropic. Moreover, let $k\in \mathbb{Z}\backslash\{0\}$ be such that $\beta+k\alpha$ is a real root. We have
    $$\beta(h_\alpha)\in\begin{cases}
    \{-k\} & \text{ if } \beta+k\alpha \text{ is isotropic},\\
   \{0,-2k\} & \text{ if } \beta+k\alpha \text{ is non-isotropic}.
\end{cases}$$
\begin{proof}
Note that 
$$\beta+k\alpha \ \text{ is isotropic } \iff (\beta+k\alpha,\beta+k\alpha)=0 \iff \beta(h_\alpha)=-k.$$ Now assume that $\beta+k\alpha$ is non-isotropic. Hence we have $\beta(h_\alpha)\neq -k$ and 
$$\alpha(h_{\beta+k\alpha})=\frac{1}{k}+\frac{1}{\beta(h_\alpha)+k}\in \bz.$$
This implies $|k|=|\beta(h_\alpha)+k|$ and the claim follows.
\end{proof} 
\end{rem}

\subsection{} In this subsection we complete the root string picture in the case when $\lie g$ is symmetrizable. Recall from Remark~\ref{finaffine} that $\lie g$ is either finite, affine or twisted affine type if $\lie g$ has at least one simple isotropic root and we will work with the explicit description of roots. Given a root of the form $\gamma+r\delta$, we call $\gamma$ its finite part.

\begin{lemma}
Let $\lie g$ be symmetrizable and $\alpha$ be a real isotropic root and $\beta$ be a root. Then $|\mathbb{S}(\beta,\alpha)\cap\Delta^{\mathrm{re}}|\le 2.$ 
Moreover, if the finite part of $\beta$ is not a scalar multiple of the \textit{finite part} of $\alpha,$ then $\mathbb{S}(\beta,\alpha)\subseteq\{\beta-\alpha,\beta,\beta+\alpha\}.$
\end{lemma}
\begin{proof} Since $\alpha$ is real isotropic, $\Delta$ must contain also a simple isotropic root. The proof will be a case-by-case analysis using the explicit description of $\Delta$ (see discussion above). We can assume that $\beta$ and $\alpha$ are not scalar multiples of each other; otherwise the statement is clear from \eqref{scalarmultiple}. We outline the proof for the case $A(2k,2\ell)^{(4)}$. The set of roots is given by 
\begin{equation}\label{rootstwisted}
    \begin{aligned}
        \Delta=& \{\pm\epsilon_i\pm\epsilon_j+2r\delta, \pm\delta_p\pm\delta_q+2r\delta, \pm \epsilon_i+r \delta, \pm\delta_{p}+r\delta, \pm 2\epsilon_i+(4r+2)\delta, \pm 2\delta_{p}+4r \delta\\ 
     & \pm\epsilon_i\pm\delta_p+2r\delta:r\in \bz, 1\le i\neq j\le k, 1\le p\neq q\le \ell\}\cup(\bz\backslash \{0\})\delta.
    \end{aligned}
\end{equation}
where the roots of the form $\pm\epsilon_i\pm\delta_p+2r\delta$ are isotropic. Let $\alpha=\lambda_i\epsilon_i+\lambda_p\delta_p+2r\delta$ be an isotropic root with $\lambda_i,\lambda_p\in \{\pm 1\}.$ If $\beta$ is an imaginary root and $\beta+t\alpha\in \Delta$ for some $t\neq 0$, then using \eqref{rootstwisted}  we obtain that $t\in \{\pm 1\}$ which gives the claim of the lemma in this case. Similarly, a direct calculation shows the claim if the finite part of $\beta$ is a scalar multiple of the finite part of $\alpha$. So assume for the rest of the proof that $\beta$ is real and the finite part of $\beta$ is not a scalar multiple of the finite part of $\alpha$.

In this situation, it will be enough to show $\mathbb{S}(\beta,\alpha)\subseteq\{\beta-\alpha,\beta,\beta+\alpha\}$ by Corollary~\ref{finaffine}(3).

\textbf{Case 1:} Let $\beta$ be of the form $\beta=\mu_s\epsilon_s+\mu_t\epsilon_t+a\delta\in \Delta$ for some $1\le s\neq t\le k$, $a\in \bz$ and $\mu_s,\mu_t\in \{0,\pm 1,\pm 2\}.$ We assume without loss of generality that $\mu_s\neq 0.$ Let $m\in \bz\backslash\{0\}$ such that
$$\beta+m\alpha=\mu_s\epsilon_s+\mu_t\epsilon_t+m\lambda_i\epsilon_i+m\lambda_p\delta_p+(a+2mr)\delta\in \Delta.$$
Since $\lambda_p\neq 0,$ the description \eqref{rootstwisted} guarantees that $|m|\le 2.$ If $\mu_s\epsilon_s+\mu_t\epsilon_t+m\lambda_i\epsilon_i\neq 0$, then we must have $m\lambda_p=\pm 1$, which forces $|m|\le 1.$ If $\mu_s\epsilon_s+\mu_t\epsilon_t+m\lambda_i\epsilon_i= 0,$ then $\mu_t=0$ implying $\mu_s+m\lambda_i=0$ and $\mu_s=-m\lambda_i.$ If $|m|=2,$ then $|\mu_s|=2$ and again \eqref{rootstwisted} implies $a\in 4\bz+2.$ This gives $a + 2mr \in 4\mathbb{Z}+2$, contradicting $\beta+m\alpha \in \Delta$.
Therefore, in all the cases we have $|m|\le 1.$ Similar arguments will imply $|m|\le 1$ if $\beta=\mu_p\delta_p+\mu_q\delta_q+b\delta\in \Delta$.

\textbf{Case 2:} Now assume that $\beta$ is of the form $\beta=\mu_j\epsilon_j+\mu_q\delta_q+2s\delta$ with $\mu_j,\mu_q\in\{\pm 1\}$. Again let $m\in\bz\backslash\{0\}$ be such that
$$\beta+m\alpha=\mu_j\epsilon_j+m\lambda_i\epsilon_i+\mu_q\delta_q+m\lambda_p\delta_p+2(s+mr)\delta\in \Delta.$$
Note that if $\mu_q\delta_q+m\lambda_p\delta_p=0$ or $\mu_j\epsilon_j+m\lambda_i\epsilon_i=0,$ then clearly $|m|\le 1$ (both of them cannot be zero since the finite parts are not proportional). Therefore we can assume that both the quantities are non-zero. This forces $j=i,q=p$ and $\mu_q+m\lambda_p,\mu_j+m\lambda_i\in \{\pm 1\}$, which in turn implies $|m|\le 2.$  However, if $|m|=2$, then the finite parts of $\alpha$ and $\beta$ are scalar multiples of each other. To see this, assume $m=2$. Then $\mu_q=\pm 1 -2\lambda_p$ and $\mu_j=\pm 1-2\lambda_i.$ Since $\lambda_i,\lambda_p,\mu_j,\mu_q\in \{\pm 1\},$ it follows that $\lambda_i=-\mu_j$ and $\lambda_p=-\mu_q$. The case $m = -2$ is analogous. This completes the proof of the lemma.
\end{proof}
\section{Root generated subalgebras and \texorpdfstring{$\pi$}{pi}-systems}\label{secrgs}
In this section, we study root generated subalgebras arising from $\pi$-systems and determine their sets of real roots. This will allow us to generalize \cite[Theorem~1]{idv2023root} (and hence Dynkin’s bijection mentioned in the introduction) to the super setting. The study of such subalgebras is important, as every graded embedding (more precisely, its derived algebra) is isomorphic to a root generated subalgebra defined by a $\pi$-system.
\textit{Throughout this section, we assume that $\lie g$ is a symmetrizable, quasisimple, regular Kac-Moody superalgebra.} 

\subsection{}\label{secpi} We study one of our main combinatorial object, namely the $\pi$-systems and see how they arise in the study of closed subroot systems. We begin with the following definition. For similar study for symmetrizable Kac-Moody algebras, we refer to \cite{idv2023root}.
\begin{definition}\label{defpi}
    A non-empty subset $\Sigma\subseteq \Delta^{\mathrm{re}}$ is called a \textit{$\pi$-system} if 
    \begin{enumerate}
        \item $\alpha-\beta\notin \Delta$ for all $\alpha,\beta\in \Sigma$
        \item $\alpha\notin \sum_{\beta\in\Sigma\backslash\{\alpha\}} \mathbb{Q}_{\geq 0}\beta$ for all $\alpha\in \Sigma.$
    \end{enumerate}
    Furthermore, for $S\subseteq \Delta^{\mathrm{re}}$, we define inductively 
    $$S_{0}=(S\cup2S)\cap\Delta^{\mathrm{re}},\ \ S_k=\pm\{s_\alpha(\beta):\alpha,\beta\in S_{k-1}\},\ \ k\in \bn,\ \ S_{\infty}=\displaystyle{\cup_{k=0}^\infty} S_k.$$
We say that a closed subroot system $\Psi$ admits a $\pi$-system if there exists a $\pi$-system $\Sigma\subseteq \Delta^{+}$ such that $\Psi=\Sigma_\infty.$
\end{definition}
A construction similar to $S_{\infty}$ appeared in \cite{gorelik17generalized} in the context of studying subroot systems of generalized reflection root systems.
\begin{rem}\label{rempi}
    Note that the condition in Definition~\ref{defpi}(2) is redundant for all non-isotropic roots. Indeed, if $\alpha\in\Sigma$ is non-isotropic, then for every $\beta \in \Sigma \backslash \{\alpha\}$ we have $\beta(h_\alpha) \le 0$. If $\Sigma$ contains only non-isotropic roots, then we have $\Sigma_\infty=W_\Sigma(\Sigma_0)$, where $W_\Sigma=\langle s_{\beta}: \beta \in \Sigma\rangle $.
\end{rem}
By \cite[Theorem 1]{idv2023root}, every root generated subalgebra of a symmetrizable Kac-Moody algebra is generated by a $\pi$-system in $\Delta^{+}$. Moreover, every closed subroot system admits a $\pi$-system \cite[Proposition 3.3]{idv2023root}. The following example shows that such statements fail in the super setting.
\begin{example}\label{thenonexample1}
Let $\mathring{\lie g}$ be a finite dimensional basic classical simple Lie superalgebra and $\lie g$ be the corresponding (untwisted) affine Lie superalgebra. Let $\alpha$ be an isotropic root of $\mathring{\lie g}.$ Let $\Psi=\{\pm \alpha,\pm(\alpha+\delta)\}.$ Then $\Psi$ is a closed subroot system and (modulo the centre) we have $$\lie g(\Psi)=\bc x_\alpha^\pm\oplus \bc(x_\alpha^\pm\otimes  t^{\pm 1})\oplus \bc (h_\alpha\otimes t^{\pm 1})\oplus \bc h_\alpha.$$ 
Thus, $\lie g(\Psi)$ is generated by the $\pi$-system $\Sigma'=\{-\alpha,\alpha+\delta\}$ and 
$$\Delta(\lie g(\Psi))^{\mathrm{re}}=\Psi.$$ 
However, $\Psi$ does not admit a $\pi$-system and $\lie g(\Psi)$ is never generated by a $\pi$-system $\Sigma$ contained in $\Delta^+$ (every such subset has to satisfy $|\Sigma|=1$). 
\end{example}

\subsection{} For closed subroot systems, we will construct natural subsets similar as in \cite{deodhar1989note}. Let $\Psi\subseteq \Delta^{\mathrm{re}}$ be a closed subroot system and set $\Psi^+=\Psi\cap\Delta^+.$ Consider the preorder (a relation that is reflexive and transitive) on $\Psi^+$ defined by
    $$\gamma_1\preceq \gamma_2 \iff \gamma_2=a\gamma_1+\sum_{\tau\in \Psi^+\backslash{\{\gamma_1,\gamma_2\}}}a_{\tau}\tau,\ \ \text{for some } a\in \mathbb{Q}_{>0},\ a_{\tau}\in\mathbb{Q}_{\geq 0}.$$
Let $\Pi(\Psi)$ be the set of elements satisfying $$\Pi(\Psi)=\{\alpha\in \Psi^+: \gamma\in \Psi^+\text{ and } \gamma\le \alpha\implies \gamma\in \bn\alpha\}.$$
In Example~\ref{thenonexample1} we have $\Pi(\Psi)=\{\alpha,\alpha+\delta\}$ and hence $\Pi(\Psi)$ is not a $\pi$-system in general.
However, we claim the following
\begin{equation}\label{weakpisystem}
\alpha\neq \beta\in\Pi(\Psi),\ \alpha-\beta\in\Delta\implies \alpha-\beta\in \Delta^{\mathrm{im}},\ \text{$\alpha,\beta$ isotropic}
\end{equation}
To see this, let $\alpha\neq \beta\in\Pi(\Psi)$ be such that $\alpha-\beta$ is a root. If $\alpha-\beta$ is real, is must be contained in $\Psi$ by closedness. Hence $\alpha\preceq \beta$ (if $\beta-\alpha\in \Psi^+$) or $\beta\preceq \alpha$ (if $\alpha-\beta\in \Psi^+$), which is a contradiction. Hence $\alpha-\beta$ is imaginary. Assume that $\alpha$ is non-isotropic. In this case 
Corollary~\ref{proprootstring} gives $s_\alpha(\beta)=\beta-k\alpha\in \Psi$ for some $k\in \bn$. If $s_\alpha(\beta)\in \Psi^+$ (resp. $s_\alpha(\beta)\in -\Psi^+$), then $\alpha\preceq\beta$ (resp. $\beta\preceq \alpha$). Hence we get once more a contradiction.
\begin{rem}\label{remdeodhar}
    The same proof as in \cite[Section 3]{deodhar1989note} using Remark~\ref{rempi} and Example~\ref{examnoniso} shows $\Pi(\Psi)\neq \emptyset$ if $\Psi$ contains at least one non-isotropic root and $\Pi(\Psi)_\infty=\Psi$, if all roots in $\Psi$ are non-isotropic.
\end{rem}
The next lemma shows that the first statement always holds, while the example below demonstrates that the second statement fails in general.
\begin{lemma}\label{lempipsi} For a closed subroot system $\Psi$, we have $\Pi(\Psi)\neq \emptyset.$ 
\begin{proof}
 By Remark~\ref{remdeodhar}, we can assume that all roots in $\Psi$ are isotropic. Since $\Psi$ is closed, by Lemma~\ref{lemimroots}(2), we must have $\gamma_1\pm \gamma_2\notin \Delta^{\mathrm{re}}$ for $\gamma_1,\gamma_2\in \Psi$.
 We treat the finite and affine cases uniformly by realizing a finite-type Kac-Moody superalgebra $\lie g$ as a subalgebra of its untwisted affinization $\lie g^{(1)}$. Due to the different behavior of isotropic roots in different types, we shall prove the result by case by case analysis. Let $r\in\mathbb{Z}_{+}$ be minimal such that $\alpha=\alpha_0+r\delta\in\Psi^+$, where $\alpha_0$ denotes the finite part of $\alpha$. If $\lie g$ is of finite type, then $r=0$ and $\alpha\in\Psi^+$ can be chosen arbitrarily. We claim that $\alpha\in\Pi(\Psi)$. Assume, for contradiction, that $\alpha\notin\Pi(\Psi)$. Then there exists an integer $t\geq 2$ such that $\alpha$ can be written as
\begin{equation}\label{funnyeq1}
\alpha = \sum_{i=1}^t a_i \beta_i,\quad\text{ where }\ \ a_i\in\mathbb{Q}_{>0},\ \beta_i\in\Psi^+\backslash\{\alpha\}.
\end{equation}
    If $\lie g$ is of affine type, we shall show that $\alpha-\beta_p\in\Delta^{\mathrm{im}}$ for some $1\le p\le t$.
    \begin{enumerate}[leftmargin=*]
        \item \textbf{Case 1.} Suppose that $\lie g$ is neither of type $F(4)$, $G(3)$, or $D(2,1;a)$, nor an untwisted affinization of these types. Then there exist $\lambda_i,\lambda_j\in\{\pm 1\}$ such that $\alpha_0=\lambda_i\epsilon_i+\lambda_j\delta_j$ and 
        $$\alpha=\lambda_i\epsilon_i+\lambda_j\delta_j+r\delta.$$
Since $(\alpha,\epsilon_i)\neq 0$, the set $J:=\{k: 1\le k\le t,\ (\beta_k,\epsilon_i)\neq 0\}$ is non-empty. Moreover, since each $\beta_k$ appearing in \eqref{funnyeq1} is isotropic, for every $k\in J$ there exist $c_{i_k},c_{m_k}\in\{\pm 1\}$ and $t_k\in\mathbb{Z}_{\ge 0}$ such that
    $$\beta_k=c_{i_k}\epsilon_i+c_{m_k}\delta_{m_k}+t_k\delta.$$
     Therefore, using \eqref{funnyeq1}, we obtain
     $$(\alpha,\epsilon_i)=\sum_{k\in J}a_k(\beta_k,\epsilon_i)\quad \Longrightarrow \quad \lambda_i=\sum_{k\in J}a_kc_{i_k}.$$
     If $c_{i_k}=-\lambda_i$ for all $k\in J$, then $-\sum_{k\in J}a_k=1$, a contradiction. Thus there exists $p$ such that $c_{i_p}=\lambda_i$, i.e.,
     \begin{equation}\label{betap}
         \beta_p=\lambda_i\epsilon_i+c_{m_p}\delta_{m_p}+t_p\delta.
     \end{equation} 
    Since $\alpha\pm\beta_p\notin\Delta^{\mathrm{re}}$, it follows that if $\lie g$ is of finite type, then ($t_p=0$ and) necessarily $\alpha=\beta_p$. Hence, for the remainder of this case we assume that $\lie g$ is of affine type and claim that $\alpha-\beta_p\in\Delta^{\mathrm{im}}$. It suffices to show that $\alpha-\beta_p\in\Delta$. First suppose that $\lie g$ is untwisted. Then $\alpha-\beta_p\notin\Delta$ would force $\alpha=\beta_p$, which is impossible. Thus $\alpha-\beta_p\in\Delta$ in this case. If $\lie g$ is twisted (not $A(2k,2l)^{(4)}$) and $\alpha-\beta_p\notin\Delta$, then necessarily $$j=m_p,\ c_{m_p}=-\lambda_j,\ \text{ and $r$ and $t_p$ have different parities}.$$ Therefore, $\beta_p=\lambda_i\epsilon_i-\lambda_j\delta_j+t_p\delta$, which implies $\alpha+\beta_p\in\Delta^{\mathrm{re}}$, a contradiction. Finally, let $\lie g$ be of type $A(2k,2l)^{(4)}$. If $\alpha-\beta_p\notin\Delta$, then we have $$j=m_p,\ c_{m_p}=-\lambda_j,\ \text{ and $(r,t_p)=(2r_1,2t_{p_1})$}$$ 
     for some $r_1,t_{p_1}$ of different parity. As in the previous case, this again implies $\alpha+\beta_p\in\Delta^{\mathrm{re}}$, a contradiction. Thus, in all affine cases we conclude that $\alpha-\beta_p\in\Delta$, and hence $\alpha-\beta_p\in\Delta^{\mathrm{im}}$.\smallskip
    \item \textbf{Case 2.} Now suppose that $\lie g$ is either of type $F(4)$, $G(3)$, or $D(2,1;a)$, or untwisted affinization of these types. Note that in this case, if $\beta=\beta_{0}+t\delta \in \Psi^{+}$ with $\beta_{0} \neq \pm \alpha_{0}$ and then we have $(\alpha_{0}+r\delta,\ \beta_{0}+t\delta) \neq 0$ which implies $\alpha-\beta\in \Delta^{\mathrm{re}}$ or $\alpha+\beta\in\Delta^{\mathrm{re}},$ which is a contradiction. Therefore, if $\lie g$ is of finite type, then $\Psi^+=\{\alpha_0\}$ and so $\alpha=\alpha_0\in \Pi(\Psi),$ which completes the proof if $\lie g$ is of finite type. If $\lie g$ of affine type, then we must have 
     $$\Psi^{+}\subseteq \{\,\pm \alpha_{0} + s\delta : s\in \mathbb{Z}_{+}\,\}.$$ 
    By comparing the coefficients of $\alpha_0$ in \eqref{funnyeq1}, it follows that there exists $p$ such that $\beta_p=\alpha_0+t_p\delta$, which implies $\alpha-\beta_p\in\Delta^{\mathrm{im}}$. 
    \end{enumerate}
    Therefore, we can assume from now on that $\lie g$ is of affine type and $\alpha-\beta_p\in\Delta^{\mathrm{im}}$. In the rest, we shall show that $\alpha\in\Pi(\Psi)$. Since $\alpha-\beta_p\in\Delta^{\mathrm{im}}$, we must have $\beta_p=\alpha_0+t_p\delta$ and the set
    $$J_1:=\{j:1\le j\le t,\ \beta_j=\alpha_0+t_j\delta\}$$ is non-empty.
    Comparing the coefficients of $\alpha_0$ and $\delta$ in \eqref{funnyeq1}, we obtain
    $$\sum_{k\in J_1}a_k\geq 1\ \text{ and }\ r\geq\sum_{k\in J_1}a_kt_k.$$ If $t_k>r$ for all $k\in J_1$, then
$$r\geq\sum_{k\in J_1}a_kt_k>\sum_{k\in J_1}a_kr\geq r$$ a contradiction. Hence there exists $k\in J_1$ such that $t_k<r$ (note that $t_k\neq r$ since $\beta_k\neq\alpha$). But then $\beta_k=\alpha_0+t_k\delta$ with $t_k<r$, contradicting the minimality of $r$ in the choice of $\alpha$. Therefore, $\alpha\in\Pi(\Psi)$ and this completes the proof.
\end{proof}
\end{lemma}
\begin{example}\label{thenonexample2}
    Let $\mathring{\mathfrak{g}}$ be the finite--dimensional Lie superalgebra of type $B(m,n)$, 
and let $\mathfrak{g}$ denote the corresponding untwisted affine Lie superalgebra. Fix $1\le i\neq k\le m$ and $1\le \ell \neq j\le n$ and consider
$$\Psi=\{\pm(\epsilon_i-\delta_j+6\delta),\ \pm(\epsilon_i-\delta_j+\delta),\ \pm(\epsilon_k-\delta_\ell+2\delta),\ \pm(\delta_\ell-\epsilon_k+3\delta)\}.$$
Then $\Psi$ is a closed subroot system of 
$\Delta^{\mathrm{re}}$. Moreover,
$$\Pi(\Psi)=\{\,\epsilon_i-\delta_j+\delta,\ \epsilon_k-\delta_\ell+2\delta,\ \delta_\ell-\epsilon_k+3\delta\,\}$$
is a $\pi$-system, but 
$$\Pi(\Psi)_{\infty}=\{\pm(\epsilon_i-\delta_j+\delta),\ \pm(\epsilon_k-\delta_\ell+2\delta),\ \pm(\delta_\ell-\epsilon_k+3\delta)\},$$
which does not coincide with $\Psi$.
\end{example}
\subsection{} In this subsection we address the fundamental problem of determining the real roots of root generated subalgebras defined by $\pi$-systems. We shall see in the last subsection that such subalgebras arise naturally in the context of graded embedding problems. For Kac–Moody algebras, these subalgebras have been studied extensively: a number of special cases were treated in \cite{biswas23onsymmetric, habib2024pisystems, morita1989certain, naito1992regular, roy2019maximal, viswanath08embedding}, and a complete solution in the symmetrizable setting was obtained in \cite{idv2023root}. In contrast, the analogous problem for Kac–Moody superalgebras is far less developed, with only a few works addressing special cases only, such as \cite{Nayak13Embedding,vander87regular}. We begin with the following lemma which will be needed later.
\begin{lemma}\label{lempisystem}
    Let $\Sigma\subseteq\Delta^{\mathrm{+}}$ be a $\pi$-system and let $\lie g(\Sigma)$ be the associated root generated subalgebra.
    \begin{enumerate}
        \item Every root vector of $\mathfrak g(\Sigma)$ is a linear combination of right-normed Lie words in either $\{x_\alpha^{+} : \alpha \in \Sigma\}$ or $\{x_\alpha^{-} : \alpha \in \Sigma\}$. In particular, every root of $\mathfrak g(\Sigma)$ is an integral linear combination of elements of $\Sigma$ with coefficients of the same sign. \vspace{0,1cm}
        \item Let $\Psi=\Delta(\lie{g}(\Sigma))^{\mathrm{re}}$. Then we have $\Sigma=\Pi(\Psi).$ \vspace{0,1cm}
         \item Let $\alpha \in \Sigma$ be non-isotropic and let $\beta \in \Delta(\mathfrak g(\Sigma))^{\mathrm{re},+}$. \vspace{0,1cm}
        \begin{enumerate}
            \item Assume that either $\bigl|\{\beta - k\alpha : k \in \mathbb{Z}_+\} \cap \Delta^{\mathrm{im}}\bigr| \ge 1$
    or $\bigl|\{\beta - k\alpha : k \in \mathbb{Z}_+\} \cap \Delta^{\mathrm{re}}\bigr| \ge 3.$
    Then $\beta(h_\alpha) \in \mathbb{N}$.\vspace{0,1cm}
    \item Let $\beta \notin \mathbb{N}\alpha$ and $\beta(h_\alpha) \in \mathbb{N}$. Then $s_\alpha(\beta) \in \Delta(\mathfrak g(\Sigma))^{\mathrm{re},+}$.
        \end{enumerate}
    \end{enumerate}
\end{lemma}
\begin{proof}
The proof of part (1) follows directly from the fact that $\Sigma$ is a $\pi$-system and the super Jacobi identity. 
For part (2), let $\beta \in \Psi^{+} \setminus \Sigma$ and write
$$\beta=c_1\beta_1+c_2\beta_2+\cdots+c_k\beta_k,\ \ \ c_i\in \bn,\ \beta_i\in \Sigma,\ \beta_i\neq \beta_j\  \forall i\neq j$$ 
which is possible by part (1). If $k=1$, then $c_1 = 2$, and by definition we have $\beta \notin \Pi(\Psi)$.
If $k \ge 2$, then $\beta_1 \notin \mathbb{N}\beta$ and $\beta_1 \le \beta$, which again implies $\beta \notin \Pi(\Psi)$.
Thus we obtain $\Pi(\Psi) \subseteq \Sigma$.
Now let $\beta \in \Sigma$ and suppose that there exists $\gamma \in \Psi^{+}$ with $\gamma\preceq \beta$ and $\gamma\notin \bn  \beta$. Then we can write
\begin{equation}\label{66ght}\beta=a\gamma+\sum_{\tau\in \Psi^+\backslash\{\beta,\gamma\}}a_\tau \tau,\ \ \ a,a_\tau\in \mathbb{Q}_{\geq 0},\ a\neq 0.\end{equation} Since $\Psi^+\subseteq \bz_+\Sigma$ we can replace each $\gamma$ and $\tau$ in \eqref{66ght} to obtain an expression of the form
$\beta=\sum_{\mu\in \Sigma} b_\mu \mu$ for some $b_\mu\in\mathbb{Q}_{\geq 0}$ with $b_\nu\neq 0$ for some $\nu\neq \beta$ (recall $\gamma\notin\bn\beta$). If $b_\beta\geq 1,$ then we have $$\sum_{\mu\in \Sigma\backslash\{\beta\}}b_\mu \mu+(b_\beta-1)\beta=0,$$ which is a contradiction. If $b_\beta<1,$ then we can write $$\beta=\sum_{\mu\in \Sigma\backslash\{\beta\}}\frac{b_\mu}{1-b_\beta} \mu,$$ which is again a contradiction to Definition~\ref{defpi}$(2)$. Hence $\Sigma\subseteq \Pi(\Psi)$ and the proof of part $(2)$ is completed.
Now we shall prove part $(3).$ Part (a) follows from Proposition~\ref{unbrokenrs}(1), Corollary ~\ref{proprootstring}, and Proposition~\ref{prop32}(2). For part (b) assume if possible that $-s_\alpha(\beta)\in \Delta(\lie g(\Sigma))^{\mathrm{re},+}.$ Then we have $\alpha=\frac{1}{\beta(h_\alpha)}(\beta+(-s_\alpha(\beta))).$ Using part (1) we can write $-s_\alpha(\beta)=\sum_{\gamma\in \Sigma}b_\gamma\gamma,\ b_\gamma\in\bz_+.$ Therefore, $$\left(1-\frac{b_\alpha}{\beta(h_\alpha)}\right)\alpha=\frac{1}{\beta(h_\alpha)}\left(\beta+\sum_{\gamma\in \Sigma\backslash\{\alpha\}} b_\gamma\gamma\right).$$
Note that $1-\frac{b_\alpha}{\beta(h_\alpha)}>0$ otherwise, some $\mathbb{Q}_{\geq 0}$-linear combination of positive roots is zero, a contradiction. Hence $\beta\preceq \alpha$ and part (2) implies $\beta\in\bn\alpha,$ a contradiction.
\end{proof}
The following theorem is one of the key results which determines the real roots of a root generated subalgebra by a $\pi$-system.
\begin{thm}\label{proprealroots}
    Let $\Sigma\subseteq \Delta^{+}$ be a $\pi$-system. Then we have $$\Sigma_{\infty}=\Delta(\lie g(\Sigma))^{\mathrm{re}}.$$
\end{thm}
\begin{proof}
    By Remark~\ref{remdeodhar}, it is enough to prove the result when $\lie g$ has at least one isotropic real root. Therefore, the Lie superalgebra $\lie g$ is either of finite or (twisted) affine type. In the reminder of the proof we shall show that $\Delta(\lie g(\Sigma))^{\mathrm{re}}\subseteq \Sigma_\infty$, since the reverse inclusion follows from Corollary~\ref{finaffine}. So let $\beta\in \Delta(\lie g(\Sigma))^{\mathrm{re},+}.$ We shall prove the claim by induction on the height $\mathrm{ht}(\beta)$ with respect to $\Pi$. By Lemma~\ref{lempisystem}(1), we can write
    $$x_\beta^+=[x_{\beta_{1}}^+,[x_{\beta_{2}}^+,\cdots,[x_{\beta_{k-1}}^+,x_{\beta_{k}}^+]\cdots]],$$
    $$\beta=\beta_1+\beta_2+\cdots+\beta_k,\ \ \beta_i\in\Sigma,\ 1\leq i\leq k, \ \beta_j+\cdots+\beta_k\in \Delta(\lie g(\Sigma)),\ 1\leq j\leq k$$
    and thus the induction begins when $\beta\in\Sigma$. So suppose that $\beta\notin \Sigma$ and define $$r_\beta:=\min\{j>1: \beta_{j}+\beta_{j+1}+\cdots+\beta_{k}\text{ is a real root}\}.$$ Since $\beta_{k}$ is a real root, the positive integer $r_\beta$ is well defined. From the explicit description of roots we also have $1<r_{\beta}\leq 3$.
    
\textbf{Case 1:} Let $r_{\beta}=2.$ Then we can write $\beta=\beta_{1}+\gamma$ where $\gamma=\beta_{2}+\beta_{3}+\cdots+\beta_{k}\in \Delta(\lie g(\Sigma))^{\mathrm{re},+}.$ If $\gamma=\beta_{1}$, then $\beta=2\beta_{1}\in \Sigma_0.$ Thus we can assume that $\gamma \neq\beta_{1}.$ By induction hypothesis, we have $\gamma\in \Sigma_k$ for some $k\in \bz_+.$ Moreover, since $\Sigma\subseteq \Sigma_k$ we have that $\beta_{1},\gamma\in \Sigma_k.$ If either $\beta_{1}$ or $\gamma$ is isotropic, say $\gamma$, then $\beta=s_\gamma(\beta_{1})\in\Sigma_{k+1}\subseteq\Sigma_\infty$ and we are done. Thus it remains to consider the case, when both are non-isotropic.

If $\gamma-\beta_{1}$ is not a root, then by the representation theory of $\lie{sl}_2$ or $\lie{osp}(1,2)$, we have that $\beta_{1}(h_\gamma)<0$ and $\gamma(h_{\beta_{1}})<0.$ Thus $\beta_{1}(h_\gamma)=-1$ or $\gamma(h_{\beta_{1}})=-1$ by Lemma~\ref{sumnotreal} and therefore $\beta=s_{\gamma}(\beta_{1})\in \Sigma_\infty$ or $\beta=s_{\beta_{1}}(\gamma)\in \Sigma_\infty.$ 

If $\gamma-\beta_{1}$ is a root, then by Lemma~\ref{lempisystem}(3) we have $\beta(h_{\beta_1})\in\bn$ and $s_{\beta_1}(\beta)\in \Delta(\lie g(\Sigma))^{\mathrm{re},+}.$ Therefore $\mathrm{ht}(s_{\beta_{1}}(\beta))<\mathrm{ht}(\beta)$ and thus by induction $s_{\beta_{1}}(\beta)\in \Sigma_\infty$ and so $\beta=s_{\beta_{1}}s_{\beta_{1}}(\beta)\in \Sigma_\infty.$

\textbf{Case 2:} Let $r_\beta=3.$ We can write $\beta=\beta_{1}+\beta_{2}+\nu$ where $\beta_{2}+\nu$ is an imaginary root. In this case we must have $(\beta_1,\beta_2)\neq 0$, since the explicit description of root vectors shows that 
$$x_{\beta}^+=[x_{\beta_1}^+,[x_{\beta_2}^+,x_{\nu}^+]]\in (\beta_1,\beta_2) \lie g_{\beta}.$$
If $\beta_{1}$ is non-isotropic, then using Lemma~\ref{lempisystem}(3), gives $\beta(h_{\beta_1})\in\bn$ and $s_{\beta_1}(\beta)\in \Delta(\lie g(\Sigma))^{\mathrm{re},+.}$ So the same argument as in Case 1 implies $\beta\in \Sigma_\infty$. So assume in the rest of the proof that $\beta_1$ is isotropic and thus $\beta$ has to be isotropic as well.

\textbf{Case 2.1:} Assume that $\beta_1+\beta_2\in \Delta^{\mathrm{re}}$. Hence it is a real root of $\lie g(\Sigma)$ by closedness and by induction we must have $\beta_1+\beta_2,\nu\in \Sigma_{\infty}$. If either $\beta_1+\beta_2$ or $\nu$ is isotropic, then $\beta \in \Sigma_\infty$ by applying the appropriate odd reflection. We may therefore assume that both of them are non-isotropic. If $\beta_1+\beta_2-\nu\notin \Delta$, then by Lemma~\ref{sumnotreal} we can deduce $(\beta_1+\beta_2)(h_{\nu})=-1$ or $\nu(h_{\beta_1+\beta_2})=-1$ and the claim $\beta\in \Sigma_{\infty}$ is immediate.

So assume that $\beta-2\nu=\beta_1+\beta_2-\nu$ is a root. Since $\beta$ is isotropic and $\beta-\nu$ is real, we have $\beta+\nu\notin \Delta$ by Corollary~\ref{finaffine}(3). From Remark~\ref{remintiso} we also have $\beta(h_{\nu})=2$ and hence
$$\mathbb{S}(\beta,\nu)=\{s_{\nu}(\beta)=\beta-2\nu,\beta-\nu,\beta\}\subseteq \Delta(\lie g(\Sigma))^{\mathrm{re}}.$$
Since $\mathrm{ht}(s_{\nu}(\beta))<\mathrm{ht}(\beta),$ the induction would imply the claim  if $\beta-2\nu\in \Delta^{+}$.

If $\beta-2\nu\in \Delta^{-}$, then $-s_\nu(\beta)=\nu+(\nu-\beta)\in \Delta(\lie g(\Sigma))^{\mathrm{re},+}$. Moreover, since $\beta-\nu\in \Delta^+,$ we have $\mathrm{ht}(-s_\nu(\beta))<\mathrm{ht}(\nu)<\mathrm{ht}(\beta).$ By induction hypothesis, we get $-s_\nu(\beta)\in \Sigma_\infty$ and thus $\beta=-s_\nu(-s_\nu(\beta))\in \Sigma_\infty.$

\textbf{Case 2.2:} Now assume that $\beta_1+\beta_2\notin \Delta^{\mathrm{re}}.$ If $\beta_1+\beta_2\notin \Delta$, we get ($\Sigma$ is a $\pi$-system)  
$$0=[x_{\beta_1}^-,[x_{\beta_1}^+,x_{\beta_2}^+]]=\beta_2(h_{\beta_1})x_{\beta_2}^+\implies (\beta_1,\beta_2)=0.$$
If $\beta_1+\beta_2\in \Delta^{\mathrm{im}}$, then $\beta_2$ is also isotropic and again $(\beta_1,\beta_2)=0$. So in either case $(\beta_1,\beta_2)=0$ which is a contradiction.
\end{proof}
\begin{cor}\label{cormainprop}
    \begin{enumerate}
        \item We have $\lie g(\Sigma_{\infty})=\lie g(\Sigma)$ for any $\pi$-system $\Sigma\subseteq\Delta^{+}.$
        \item Let $\Psi$ be a closed subroot system. If $\Psi$ admits a $\pi$-system $\Sigma,$ then $\Sigma=\Pi(\Sigma_\infty)=\Pi(\Psi).$
    \end{enumerate}
     \begin{proof}
           The first part is immediate from Theorem~\ref{proprealroots} and part (2) follows from Theorem~\ref{proprealroots} and Lemma~\ref{lempisystem}(2).
        \end{proof}
\end{cor}
We end this subsection by obtaining Dynkin's analogous bijections in the superalgebra setting. Define the following maps:
$$\begin{array}{cccccr}
     \vspace{10pt}
     \left\{\begin{array}{c}
            \pi\text{-systems}  \\
             \text{contained in } \Delta^+
       \end{array}\right\}  & \longleftrightarrow & \left\{\begin{array}{c}
            \text{real closed subroot}  \\
             \text{systems admitting}\\
             \text{a }\pi\text{-system}
       \end{array}\right\}& \longleftrightarrow & \left\{\begin{array}{c}
            \text{root generated}\\
             \text{subalgebras generated} \\
             \text{  by }\pi\text{-systems in }\Delta^+
       \end{array}\right\} \\ \vspace{10pt}
        \Sigma  & \longmapsto & \Sigma_\infty\\
    \Pi(\Psi) & \longmapsfrom &\Psi & \longmapsto& \lie g(\Psi)&\\ \\
    &&\Delta(\lie g(S))^{\mathrm{re}}& \longmapsfrom &\lie g(S)&
\end{array}$$
\begin{thm}\label{dynbij}
    The maps defined above are well-defined bijections.
\end{thm}
\begin{proof}
    First, we prove that the first map is a bijection. By Theorem~\ref{proprealroots} and Corollary~\ref{finaffine}, we have that $\Sigma_\infty$ is a closed subroot system and by definition, it admits a $\pi$-system $\Sigma.$ Therefore, the map is well defined. Clearly, the map is surjective and the injectivity follows from Corollary ~\ref{cormainprop}(2).

    We now prove that the second map is a bijection. Let $\Psi$ be a closed subroot system admitting a $\pi$-system $\Sigma$. Then $\Psi = \Sigma_\infty$, and by Corollary~\ref{cormainprop}(1), we have $\mathfrak g(\Sigma_\infty) = \mathfrak g(\Sigma)$, so the map is well defined and surjective. For injectivity, let $\Psi' = \Sigma'_\infty$ and $\Psi'' = \Sigma''_\infty$ be two closed subroot systems, both admitting $\pi$-systems, such that $\mathfrak g(\Psi') = \mathfrak g(\Psi'')$. By Theorem~\ref{proprealroots} and Corollary~\ref{cormainprop}(1) we get $\Sigma'_\infty = \Sigma''_\infty$, and applying Corollary~\ref{cormainprop}(2) we obtain
    $$\Sigma'=\Pi(\Sigma'_\infty)=\Pi(\Sigma''_\infty)=\Sigma''.$$
    This completes the proof.
\end{proof}
\subsection{}\label{secemb}
In this subsection, we describe the graded embedding problem for Kac--Moody superalgebras and explain how $\pi$-systems arise naturally in this setting. Since their introduction by Dynkin~\cite{dynkin52semisimple}, $\pi$-systems have played a central role in this area, and numerous graded embeddings have been constructed using them. For further details, we refer the reader to~\cite{feingold2004subalgebras, viswanath08embedding}.

In contrast, the literature on graded embeddings for Kac--Moody superalgebras remains relatively limited. In \cite{vander87regular}, the authors showed that, when $\mathfrak{g}$ is finite-dimensional, the root generated subalgebra associated with a linearly independent $\pi$-system is semisimple. Furthermore, in \cite{Nayak13Embedding}, an analogue of the result in \cite{viswanath08embedding} was established in the superalgebra setting. However, rather than working with the maximal ideal intersecting $\mathfrak{h}$ trivially, they considered only the Serre relations together with the contragredient relations. 
 
Let $\mathfrak{g}$ be symmetrizable, and fix a matrix $B$ with index set $J$ such that there exists a graded embedding, i.e., a map sending the Chevalley generators to real root vectors:
\[
f: \mathfrak{g}(B) \hookrightarrow \mathfrak{g}, \quad y_j^\pm \mapsto x_{\beta_j}^\pm, \quad j \in J,
\]
where $\Sigma := \{\beta_j : j \in J\} \subseteq \Delta^{\mathrm{re}}$, and $y_j^\pm$ denote the Chevalley generators of $\mathfrak{g}(B)$.

\begin{lemma}
We have that $\Sigma= \{\beta_j : j \in J\}$ is a linearly independent $\pi$-system and $f$ induces an isomorphism $\tilde{f}: \lie g'(B)\to \lie g(\Sigma)$.
    \begin{proof}
        From the injectivity and Corollary~\ref{finaffine} we obtain that $\Sigma$ is a linearly independent $\pi$-system. The restriction of $f$ to $\lie g'(B)$ remains injective and the image is given by the root generated subalgebra $\lie g(\Sigma)$.
    \end{proof}
\end{lemma}
Thus the derived algebra of any such embedding is generated by some $\pi$-system of $\lie g$. 
\begin{rem}
     Although not every root generated subalgebra in the super setting is generated by a $\pi$-system contained in $\Delta^+$, this is the case in many examples of graded embeddings. It would be interesting to establish a general proof of this phenomenon. If $\lie g$ is a symmetrizable Kac-Moody algebra, then this fact has been proven in \cite[Proposition 3.4]{idv2023root}, i.e., there is a $\pi$-system $\Sigma'\subseteq\Delta^{\mathrm{re},+}$ with  $\lie g'(B)\cong \lie g(\Sigma)=\lie g(\Sigma').$
\end{rem}
The graded embedding problem has also various applications in physics, see for example \cite{damour03cosmological, feingold83ahyperbolic, superspace81}. For instance, the real roots and representations of the hyperbolic Kac-Moody algebra $ E_{10} $ are known to correspond to the fields of 11-dimensional supergravity at low levels \cite{damour03cosmological} and the graded embedding of $AE_3$ into $E_{10}$ (proven in \cite{viswanath08embedding}) reflects the embedding of Einstein gravity into 11-dimensional supergravity. An alternative embedding of $AE_3$ into $ E_{10} $ was obtained through the gravity truncation method described in \cite{damour09fermionic}. 

In the Gaiotto--Witten Janus configuration (\cite{GaiottoWitten2010Janus}), the authors show how a four-dimensional $\mathcal{N}=4$ super Yang--Mills theory with a spatially varying coupling (including the theta-angle) gives rise, at the interface, to a three-dimensional  \(\mathcal{N}=4\) Chern--Simons theory. This can be viewed algebraically as embedding the smaller 3d supersymmetry (and its associated superalgebra) into the larger 4d superalgebra across the interface, realizing a graded embedding of lower-dimensional supersymmetry in a higher-dimensional theory. In \cite{BershadskyOoguri1989Hidden} the authors showed that a graded subalgebra of $\lie{osp}(n,2)$ acts as a spectrum-generating symmetry on the space of physical fields in certain superconformal field theories. The constraints of this hidden symmetry are then used to compute correlation functions.

\bibliographystyle{plain}
\bibliography{ref}
\end{document}